\documentclass[10pt, leqno]{amsart}

\usepackage{amsmath,amssymb,amsthm, epsfig}
\usepackage[mathscr]{eucal}
\usepackage{mathptmx}
\usepackage{hyperref}
\usepackage{stackrel}

\usepackage{epsfig}

\title{Sharp and improved regularity for a class of doubly degenerate parabolic PDEs}
\author{Jo\~{a}o Vitor da Silva, \quad Elzon C. J\'{u}nior \quad and \quad Gleydson C. Ricarte}

\def \R {\mathbb{R}}

\def \Div {\mathrm{div}}

\def \loc {\mathrm{loc}}

\newcommand{\defeq}{\mathrel{\mathop:}=}

\newtheorem{theorem}{Theorem}[section]
\newtheorem{lemma}[theorem]{Lemma}
\newtheorem{proposition}[theorem]{Proposition}
\newtheorem{corollary}[theorem]{Corollary}

\theoremstyle{definition}

\newtheorem{definition}[theorem]{Definition}
\newtheorem{example}[theorem]{Example}

\theoremstyle{remark}

\newtheorem{remark}[theorem]{Remark}

\numberwithin{equation}{section}

\newcommand{\intav}[1]{\mathchoice {\mathop{\vrule width 6pt height 3 pt depth  -2.5pt
\kern -8pt \intop}\nolimits_{\kern -6pt#1}} {\mathop{\vrule width
5pt height 3  pt depth -2.6pt \kern -6pt \intop}\nolimits_{#1}}
{\mathop{\vrule width 5pt height 3 pt depth -2.6pt \kern -6pt
\intop}\nolimits_{#1}} {\mathop{\vrule width 5pt height 3 pt depth
-2.6pt \kern -6pt \intop}\nolimits_{#1}}}

\begin{document}

\begin{abstract}
In this manuscript we establish local $C^{\alpha, \frac{\alpha}{\theta}}$ regularity estimates for bounded solutions of a certain class of doubly degenerate evolution PDEs, whose simplest model case is given by
$$
    \frac{\partial u}{\partial t}-\Div(m\lvert u\rvert^{m-1}\lvert \nabla u \rvert^{p-2}\nabla u) = f(x, t) \quad \text{in} \quad \Omega_T \defeq \Omega \times (0, T),
$$
for $m\ge 1$, $p\ge 2$ and $f$ belonging to a suitable anisotropic Lebesgue space. By making use of intrinsic scaling techniques and geometric tangential methods, we derive sharp regularity estimates for such models, which depend only on universal and compatibility parameters of the problem. In such a scenario, our results are natural improvements for former ones in the context of nonlinear evolution PDEs with degenerate structure via a unified approach. As a consequence for our findings and approach, we address a Liouville type result for entire solutions of a related homogeneous problem with frozen coefficients and asymptotic estimates under a certain approximating regime, which may have their own mathematical interest. We also deliver explicit examples of degenerate PDEs where our results take place.

\medskip

\noindent \textbf{Keywords:} Sharp and improved H\"{o}lder regularity, Doubly degenerate parabolic PDEs, intrinsic scaling techniques, Geometric tangential analysis.

\medskip

\noindent \textbf{AMS Subject Classifications:} 35B65, 35K55, 35K65.
\end{abstract}

\maketitle



\section{Introduction}\label{s1}

Throughout this work we will address sharp and improved (geometric) $C_{\text{loc}}^{\alpha, \frac{\alpha}{\theta}}$ regularity estimates for bounded weak solutions to doubly degenerate parabolic equations, whose prototype is given by models in divergence form (with structure, to be clarified \textit{a posteriori})
\begin{equation}\label{1.1}
  \mathcal{Q}_{\mathcal{A}} u \defeq \frac{\partial u}{\partial t}-\Div(\mathcal{A}(x, t,u, \nabla u)) = f(x, t) \quad \text{in} \quad \Omega_T,
\end{equation}
where the source term $f$ belongs to suitable Lebesgue spaces with mixed norms (cf. \cite{BP}) and $\Omega_T \defeq \Omega\times(0,T)$ with $\Omega\subset\R^n$ an open, bounded set and $T>0$.

The prototype which we have in mind in \eqref{1.1} is the following doubly nonlinear degenerate model, namely the inhomogeneous porous-medium $p$-Laplacian (or the so-named $(m,p)-$Laplacian) equation given by
\begin{equation}\label{1.2}
  \mathcal{Q}_{m, p} u \defeq  \frac{\partial u}{\partial t}-\Div(m\lvert u\rvert^{m-1}\lvert \nabla u \rvert^{p-2}\nabla u) = f(x, t) \quad \text{in} \quad \Omega_T,
\end{equation}
and
$$
  \hat{\mathcal{Q}}_{m, p} u \defeq  \frac{\partial u}{\partial t}- \sum_{i=1}^{n}  \frac{d}{dx_i} \left\{ m|u|^{m-1}\left|\frac{\partial u}{\partial x_i}\right|^{p-2}\frac{\partial u}{\partial x_i} \right\} = f(x, t) \quad \text{in} \quad \Omega_T,
$$
In this framework, our investigations have a physical-mathematical's appeal, for instance,
in the analysis of the filtration of a polytropic non-Newtonian fluid in a porous medium. Such equations also arise in the study of turbulent filtration of a gas or a fluid through porous media, in theoretical glaciology, plasma physics, image-analysis and ground water problems (see \cite{AMS04}, \cite{Iva97}, \cite{Iva00} and therein references). In addition, it should be pointed out that for $m\geq 1$, $p\geq 2$, our model case extends the well-known porous medium equation (case $p=2$ -- cf. \cite{AMU20} and \cite{Diehl21}), as well as the evolution $p$-Laplacian equation (case $m=1$ -- cf. \cite{TU14}). Finally, for $m=1$ and $p=2$ our researches recover the classical estimates to heat operator without making use of energy considerations (cf. \cite{daST17} and \cite{K08}). In summary, our approach is stable as $m \to 1^{+}$ and $p \to 2^{+}$, which allow us to include the endpoint cases $m=1$ and/or $p=2$.

Additionally, concerning \eqref{1.2}, such a model case possesses a degeneracy law driven by a double nonlinearity, whose ``modulus of ellipticity'', i.e. $|u|^{m-1}|\nabla u|^{p-2}$, collapses along vanishing and singular points of existing solutions, namely
$$
\mathcal{Z}(u) \defeq \{u=0\} \quad \text{and} \quad \mathcal{S}(u) \defeq \{|\nabla u|=0\}.
$$
 Furthermore, the presence of such degeneracy law suggests the use of intrinsic scaling and geometric tangential techniques adapted to our context. For this reason, one needs to consider several new aspects in the original argument presented, for instance, in \cite{AdaSRT}, \cite{JVSilva19} and \cite{TU14} in the scenario of evolutionary $p-$Laplacian type equations and \cite{Ara20} for the corresponding doubly degenerate model.

In our researches, we are looking for quantitative features for weak solutions which depends only on structural and universal parameters of the problem\footnote{Throughout this manuscript universal and structural read that the corresponding constants depend only on degeneracy parameters $m$ and $p$, dimension $n$, \text{a priori} estimates of the homogeneous problem with frozen coefficients, and bounds for configurational properties of $\mathcal{A}$, i.e. $\mathrm{C}_1$, $\mathrm{C}_2$, $\gamma_{1}\leq \gamma_{2}$, $\psi_{1}\leq \psi_{2}$ and $\sigma_0$. Please, see Subsection \ref{Assump} for more details}. Furthermore, it is worth highlighting the importance of this type of information, which play a decisive role in several mathematical contexts: blow-up analysis, related analysis in geometric and free boundary problems, analysis of asymptotic behavior of certain solutions, and for establishing Liouville type results, just to mention a few examples (cf. \cite{DOS18} and \cite{TU142}).

Particularly, we have interest in sharp and improved H\"{o}lder regularity estimates to locally bounded solutions of models like \eqref{1.1}. Such issues were our main impulse in studying non-linear evolution PDEs with doubly degenerate structure via a modern and systematic approach based on geometric regularity methods and intrinsic scaling techniques. As a matter of fact, the insights that were of paramount importance in driving our results have been inspired in techniques from regularity theory of nonlinear equations and free boundary problems (cf. \cite{Tei16} and \cite{TU18}, \cite{U08} for instrumental surveys concerning classical and modern results on geometric regularity methods and intrinsic scaling techniques in the elliptic and parabolic sceneries).

Finally, it is noteworthy that our contributions extend (regard to $C^{0,\alpha}$ scenario), as well as improve, to some extent, former seminal results (sharp regularity estimates) from Ara\'{u}jo \textit{et al} \cite[Theorem 6]{AMU20}, Ara\'{u}jo \cite[Theorem 1.1]{Ara20}, Diehl \cite[Theorem 2.5]{Diehl21} and Teixeira-Urbano \cite[Theorem 3.4]{TU14} by making use of intrinsic scaling techniques and a geometric tangential approach adjusted to our doubly nonlinear setting in a unified fashion.

\subsection{Background assumptions and statement of the main result}\label{Assump}

In this part we present the structural properties under $\mathcal{A}: \Omega_T \times \R \times \R^{n} \to \R^n$ and $f: \Omega_T \to \mathbb{R}$, which we will work with (for $m\ge 1$ and $p\geq 2$):
\smallskip
\begin{itemize}
	\item[(P1)][{\bf Degenerate ellipticity}]. There exists a positive constant $\mathrm{C}_1$ such that for a.e. $(x,t) \in \Omega_{T}$ there holds
	$$
		\langle \mathcal{A}(x, t,u, \nabla u), \nabla u \rangle \geq \mathrm{C}_{1}\Phi(|u|)|\nabla u|^p.
	$$
\smallskip
	\item[(P2)][{\bf Growth condition}]. There exists a positive constant $\mathrm{C}_{1}$ such that for a.e. $(x, t)\in\Omega_T$ there holds
	$$
		|\mathcal{A}(x, t, u, \nabla u)| \leq \mathrm{C}_{2}\Phi(|u|)|\nabla u|^{p-1}.
	$$
where $\Phi:\R^{+}\to \R^{+}$ is continuous function, which satisfies
$$
\gamma_{1}s^{m-1}\leq \Phi(s)\leq \gamma_{2}s^{m-1}, \forall s \in [0,\sigma_{0}]
$$
and
$$
\psi_{1}\leq \Phi(s)\leq \psi_{2} \quad, \forall s \in (\sigma_{0}, \infty)
$$
for positive constants $\gamma_{1}\leq \gamma_{2}$, $\psi_{1}\leq \psi_{2}$ and $\sigma_0$.
\smallskip
    \item[(P3)][{\bf{ Oscillation of the coefficients}}] There exist a modulus of continuity $\omega_{\mathcal{A}}: [0, \infty) \to [0, \infty)$ and a universal constant $\mathrm{C}_{\mathcal{A}}>0$ such that
    $$
\begin{array}{rcl}
  \Theta_{\mathcal{A}}(x, t, x_0, t_0) & \defeq & \displaystyle \sup\limits_{\xi \ne 0,s\ne 0}\frac{|\mathcal{A}(x,t,s,\xi)-\mathcal{A}(x_{0},t_{0},s,\xi)|}{|s|^{m-1}|\xi|^{p}} \\
   & \le & \mathrm{C}_{\mathcal{A}}\cdot \omega_{\mathcal{A}}(|(x, t)-(x_0, t_0)|)
\end{array}
    $$
\smallskip
\item[(P4)][{\bf{ Integrability of the source term}}] The source term $f \in L^{q, r}(\Omega_T) = L^r(0, T; L^q(\Omega))$, i.e., a Lebesgue space with mixed norms, (cf. \cite{BP}), which is a Banach space endowed with the norm:
$$
  \|f\|_{L^{q, r}(\Omega_T)} \defeq \left(\int_{0}^{T} \left(\int_{\Omega}|f(x, t)|^qdx\right)^{\frac{r}{q}}dt\right)^{\frac{1}{r}}.
$$
\end{itemize}
Furthermore, we will assume the following weaker compatibility conditions:
\begin{equation}\label{w-cc}{\tag{W-CC}}
\frac{1}{r}+\frac{n}{pq}<1 < \frac{2}{r}+\frac{n}{q},
\end{equation}
which implies (due to range of the parameters) the stronger compatibility conditions
\begin{equation}\label{cc}{\tag{S-CC}}
\frac{1}{r}+\frac{n}{pq}<1 \quad \text{and} \quad  \frac{3}{r}+m\left(1-\frac{1}{r}\right)+\frac{n}{q}>2 \quad (\text{for} \quad q, r>1).
\end{equation}

We want to stress that such assumptions (i.e. \eqref{w-cc} -- resp. \eqref{cc}) provide the minimal integrability condition, guaranteing the existence of bounded weak solutions of \eqref{1.1}, and it gives access to Caccioppoli estimates (see e.g. \cite[Ch.2, \S 1]{D93} and \cite[Proposition 3.1]{FS08}), as well as it defines the fashion in which weak solutions enjoy a genuine H\"{o}lder continuous modulus of continuity.

Let us now comment on the literature related to doubly degenerate PDEs: It is worth pointing out that, under structural assumptions (P1)-(P3) and taking into account the compatibility conditions \eqref{w-cc}, existence and uniqueness of solutions in suitable Sobolev spaces were established in \cite{IvaJag00}, \cite{Stur17} and \cite{Stur17-2}. Regularity and local behaviour of local weak solutions of doubly degenerate evolution models \eqref{1.1} received an increasing focus in the last decades (cf. \cite{BS99}, \cite{FSV14}, \cite{Iva89}, \cite{Iva91}, \cite{Iva95}, \cite{Iva97}, \cite{Iva98} and \cite{Iva00}, \cite{PV93} and \cite{Surn14}) due to their intrinsic connection to many problems arising nonlinear potential theory,  non-Newtonian fluids, mathematical physics, etc (cf. \cite{AMS04}, \cite{BDM13}, \cite{DiBUV02} and \cite{U08} for complete essays on regularity of evolution equations with degenerate diffusion, and \cite{KM11} and \cite{Stur18} for pioneering works by making use of parabolic potential estimates).

In spite of the fact that weak solutions of \eqref{1.1} under the compatibility assumption \eqref{w-cc} are known to be locally of the class $C^{0,\alpha}$ (in the parabolic sense) for some $\alpha \in (0, 1)$, the sharp exponent is known only for some specific sceneries (see, \cite{K08} for the linear case, \cite{AC86} and \cite{GS16} for the porous medium equation, \cite{TU14} for the inhomogeneous evolutionary $p-$Laplacian and \cite[Sec. 8]{GS16}, \cite[p. 2012]{Iva95} and \cite[Remark 3.4]{Iva00} for the doubly degenerate homogeneous model with frozen coefficients). Finally, this type of quantitative information is essential in several contexts from Mathematical Analysis, Geometry and Free Boundary Problems (cf. \cite{DOS18}, \cite{daSRS18}, \cite{daSS18} and \cite{TU142} and for some enlightening examples). In conclusion, this provided one of the main impetus for our decision in investigating \eqref{1.1}.

Recently, Ara\'{u}jo in \cite[Theorem 1.1]{Ara20} proved that, for $\Omega \subset \R^n$ an open and bounded domain, $\Omega_T \defeq \Omega \times (0, T)$, and $m >1$ and $p>2$, weak solutions of
$$
    \frac{\partial u}{\partial t}-\Div(m\lvert u\rvert^{m-1}\lvert \nabla u \rvert^{p-2}\nabla u) = f(x, t) \quad \text{in} \quad \Omega_T,
$$
belong to $C^{\alpha, \frac{\alpha}{\theta}}_{\text{loc}}(\Omega_T)$, where
\begin{equation}\label{AraujoExponent}
  \alpha \defeq \min\Big\{\frac{\alpha^{-}_{\mathrm{Hom}}(p-1)}{m+p-2},\frac{(pq-n)r-pq}{q[(r-1)(m+p-2)+1]}\Big\}
\end{equation}
and
$$
\theta \defeq p -\alpha(m+p-2).\Big(1-\frac{1}{m+p-2}\Big)
$$
under the compatibility assumptions
$$
\frac{1}{r}+\frac{n}{pq}<1 \quad \text{and} \quad \frac{3}{r}+\frac{n}{q}>2,
$$
which, the second one, is different from ours in \eqref{w-cc} and \eqref{cc}.

In this manuscript, we will derive the sharp and improved exponent (in some scenarios and under sharp integrability conditions \eqref{w-cc} and respectively \eqref{cc}) for the H\"{o}lder regularity of weak solutions of \eqref{1.1} under the general assumptions (P1)-(P4). Precisely, we find the sharp and improved exponent
{\scriptsize{
\begin{equation}\label{Sharp_alpha}{\tag{Sharp}}
    \alpha \defeq \min\left\{\max\left\{\frac{\alpha^{-}_{\mathrm{Hom}}p}{p_m+\alpha_{\mathrm{Hom}}(m+p-3)}, \frac{2\alpha^{-}_{\mathrm{Hom}}(p-1)}{p_m(m+p-2)}\right\}, \frac{(pq-n)r-pq}{q[(r-1)(m+p-2)+1]}\right\},
\end{equation}}}
where $\alpha_{\mathrm{Hom}} \in (0, 1]$ is the optimal H\"{o}lder exponent to homogeneous problem with frozen coefficients, $\iota^{-}$ means that we can select any value such that $s \in (0, \iota)$, and
\begin{equation}\label{p_m}{\tag{$p_m$}}
p_m \defeq
\left\{
\begin{array}{rcl}
  2 & \text{if} & m=1 \\
  p & \text{if} & m>1.
\end{array}
\right.
\end{equation}

Now, before enunciating our main result, we need to define the parabolic cylinders where our estimates will take place:
$$
  Q^{-}_{\rho}(x_0, t_0)  = B_{\rho}(x_0) \times \left(t_0-\rho^{\theta}, t_0\right],
$$
where $\theta>0$ is \textit{intrinsic scaling factor} given by
\begin{equation}\label{Int-Scal}
    \theta \defeq  p-\alpha(m+p-3) = p -\alpha(m+p-2).\Big(1-\frac{1}{m+p-2}\Big),
\end{equation}
and $\alpha>0$ is the optimal H\"{o}lder exponent given by \eqref{Sharp_alpha}. Moreover, notice that
$$
1+\frac{p-1}{p+m-2}\le \theta \le p \quad \text{for any} \quad p \ge 2 \quad \text{and} \quad m \ge 1.
$$

Taking into account the previous definitions, the following result holds:

\begin{theorem}\label{t1.2}
Let $u$ be a bounded weak solution of \eqref{1.1} in $Q^{-}_1$ and suppose that (P1)-(P4) are in force. Then, $u \in C_{\text{loc}}^{\alpha, \frac{\alpha}{\theta}}(Q^{-}_1)$ (in the parabolic sense), i.e., there exists a constant $\mathrm{M}_0(\verb"universal")>0$ such that
$$
  \displaystyle  [u]_{C^{\alpha, \frac{\alpha}{\theta}}\left(Q^{-}_{\frac{1}{2}}\right)} \leq \mathrm{M}_0.\left[\|u\|_{L^{\infty}(Q^{-}_1)} + \|f\|_{L^{q, r}(Q^{-}_1)}\right],
$$
where $\alpha \in (0, 1)$ is defined by \eqref{Sharp_alpha}, $\theta$ is given by \eqref{Int-Scal}, and
$$
   [u]_{C^{\alpha, \frac{\alpha}{\theta}}\left(Q^{-}_{\frac{1}{2}}\right)} \defeq \sup_{0< \rho \leq \rho_0} \left(\inf_{(x_0, t_0) \in Q^{-}_{\rho_0}} \frac{\|u - u(x_0, t_0)\|_{L^{\infty}\left(Q^{-}_{\rho_0}(x_0, t_0) \cap Q^{-}_{\frac{1}{2}}\right)}}{\rho_0^{\alpha}}\right).
$$
\end{theorem}

To the best of the authors' knowledge, this is the first time that inhomogeneous problems with variable coefficients for this class of doubly degenerate operators can be treated in a unified approach and under a spectrum of general assumptions.

Mathematically, the former result states that if the associated homogeneous equation with constant coefficients enjoys a good regularity theory. Then, weak solutions to \eqref{1.1} inherits some regularity estimates, provided the oscillation of the coefficients is under control in a continuous fashion and the source term obeys an appropriate integrability regime (see Lemmas \ref{l2.1} and \ref{1stStepInduc} for details).

Now, we notice that
\begin{equation}\label{AsymAlpha}
\mathrm{M}_{\sharp}< \alpha_{\mathrm{Hom}} \quad \text{for any} \quad p>2 \quad \text{and} \quad m>1,
\end{equation}
where
$$
\mathrm{M}_{\sharp} \defeq \max\left\{\frac{\alpha^{-}_{\mathrm{Hom}}p}{p_m+\alpha_{\mathrm{Hom}}(m+p-3)}, \frac{2\alpha^{-}_{\mathrm{Hom}}(p-1)}{p_m(m+p-2)}\right\}.
$$

Moreover, we must understand \eqref{Sharp_alpha} as follows:
{\scriptsize{
$$
\left\{
\begin{array}{cccccl}
  \text{If} & \mathrm{M}_{\sharp}\le \frac{(pq-n)r-pq}{q[(r-1)(m+p-2)+1]} & \text{then} & u \in C_{\text{loc}}^{\alpha, \frac{\alpha}{\theta}} & \text{for any} & \alpha < \mathrm{M}_{\sharp};\\
  \text{If} & \mathrm{M}_{\sharp} > \frac{(pq-n)r-pq}{q[(r-1)(m+p-2)+1]} & \text{then} & u \in C_{\text{loc}}^{\alpha, \frac{\alpha}{\theta}} & \text{for} & \alpha = \frac{(pq-n)r-pq}{q[(r-1)(m+p-2)+1]}.
\end{array}
\right.
$$}}

Regarding the optimal value of $\alpha_{\mathrm{Hom}}$, the Barenblatt function
$$
\mathcal{B}_{m, p}(x, t) \defeq \left\{
\begin{array}{cc}
 \frac{1}{t^{\lambda_0}}\left[1-b(m, n, p)\left(\frac{|x|}{t^{\frac{1}{\lambda_0}}}\right)^{\frac{p}{p-1}}\right]_{+}^{\frac{p-1}{m+p-3}} & t>0 \\
 0 & t \le 0
\end{array}
\right.
$$
is the fundamental solution for equation
$$
  \frac{\partial u}{\partial t}-\Div(m\lvert u\rvert^{m-1}\lvert \nabla u \rvert^{p-2}\nabla u) = 0,
$$
where
$$
\lambda_0 = n(m+p-3) \quad \text{and} \quad b(m, n, p) = \frac{p-1}{p}\frac{m+p-3}{(m+p-2)\lambda_0^{\frac{1}{p-1}}},
$$
which suggests (see \cite[Sec. 8]{GS16}, \cite[p. 2012]{Iva95} and \cite[Remark 3.4]{Iva00}) that such an optimal H\"{o}lder exponent should be
\begin{equation}\label{SharpAlpha_Hom}
  \alpha_{\mathrm{Hom}} = \min\left\{1, \frac{p-1}{m+p-3}\right\}.
\end{equation}

In turn, we can re-write the second exponent in \eqref{Sharp_alpha} as follows
{\scriptsize{
$$
\frac{(pq-n)r-pq}{q[(r-1)(m+p-2)+1]} = \frac{p\left[1-\left(\frac{n}{pq}+\frac{1}{r}\right)\right]}{p\left[1-\left(\frac{n}{pq}+\frac{1}{r}\right)\right] + \left\{\left[\frac{3}{r}+m\left(1-\frac{1}{r}\right)+\frac{n}{q}\right]-2\right\}} \in (0, 1)
$$}}
by using the compatibility conditions \eqref{cc}.

It is also important to stress that
$$
\alpha \to \min\left\{\mathrm{M}_{\sharp}, \frac{p}{m+p-2}\right\} \quad \text{as} \quad q, r \to \infty.
$$
Moreover, $\frac{p}{m+p-2} \geq \frac{p-1}{m+p-3}$ if and only if $m\geq 2$.

In this context, by using \eqref{AsymAlpha} and \eqref{SharpAlpha_Hom} we conclude that
$$
0 <\alpha \leq \frac{p-1}{m+p-3}  \quad \text{as} \quad q = r = \infty \,\,\, \text{for any} \,\,\, p>2 \,\,\, \text{and} \,\,\, m\ge 2.
$$
\bigskip

In the sequel, we present an overview on recent regularity results which our H\"{o}lder estimates in Theorem \ref{t1.2} recover, as well as, to some extent, improve in a unified fashion.

\begin{table}[h]
\centering
\resizebox{\textwidth}{!}{
\begin{tabular}{|c|c|c|c|c|}
\cline{1-4}
 \textbf{Model equation}&{\textbf{Compatibility conditions}} & {\bf Sharp H\"{o}lder regularity} &{\textbf{References}}  \\\cline{1-4}
 $\frac{\partial u}{\partial t}-\Delta u=f$ & $1< \frac{n}{q} + \frac{2}{r}< 2$ & $\alpha=2-\big(\frac{2}{r}+\frac{n}{q}\big)$& \cite{daST17} \,\,\text{and}\,\, \cite{K08} \\\cline{1-4}
 $\frac{\partial u}{\partial t}-\Delta u^{m}=f$&$\frac{1}{r}+\frac{n}{2q}<1$&$\alpha=\min\Big\{\frac{\alpha^{-}_{\mathrm{Hom}}}{m},\frac{[(2q-n)r-2q]}{q[mr-(m-1)]}\Big\}$& \cite{AMU20}\\\cline{1-4}
 $\frac{\partial u}{\partial t}-\Delta u^{m}=f$ & $\frac{1}{r}+\frac{n}{2q} <1$ & $\alpha=\min\Big\{\frac{2\alpha^{-}_{\mathrm{Hom}}}{2+(m-1)\alpha_{\mathrm{Hom}}},\frac{(2q-n)r-2q}{q[mr-(m-1)]}\Big\}$ & \cite{Diehl21} \\\cline{1-4}
 $\frac{\partial u}{\partial t}-\Delta_{p} u=f $ & $\frac{n}{q} + \frac{2}{r}<1< \frac{1}{r}+\frac{n}{pq}$ & $\alpha=\frac{(pq-n)r-pq}{q[(p-1)r-(p-2)]}$ & \cite{TU18} \\\cline{1-4}
$\frac{\partial u}{\partial t}-\Div{(m\lvert u\rvert^{m-1}\lvert \nabla u\rvert^{p-2}\nabla u)}=f $ & $\frac{1}{r}+\frac{n}{pq}<1 \,\, \text{and} \,\, \frac{3}{r}+\frac{n}{q}>2$ & $\alpha=\min\Big\{\frac{\alpha^{-}_{\mathrm{Hom}}(p-1)}{m+p-2},\frac{(pq-n)r-pq}{q[(r-1)(m+p-2)+1]}\Big\}$ & \cite{Ara20} \\\cline{1-4}
\end{tabular}}
\end{table}

\bigskip

In this final part, we will present a number of scenarios where our results improve the former Ara\'{u}jo's result in \cite[Theorem 1.1]{Ara20} among other ones. For this end, since
\begin{equation}\label{EqExpoentes}
  \frac{2(p-1)}{p_m(m+p-2)} \le \frac{p-1}{m+p-2} \quad \text{for any} \quad p\geq 2 \quad \text{and} \quad m\ge 1
\end{equation}
then, we must impose (for $m>1$) the following algebraic condition:
\begin{equation}\label{CImprov}
  \frac{\alpha^{-}_{\mathrm{Hom}}p}{p+\alpha_{\mathrm{Hom}}(m+p-3)} \ge \frac{\alpha^{-}_{\mathrm{Hom}}(p-1)}{m+p-2} \quad \Leftrightarrow \quad \alpha_{\mathrm{Hom}} \leq \frac{m-1}{p-1}.\frac{p}{p+m-3}.
\end{equation}
Such a condition brings to light by comparing the regularity exponents in \eqref{AraujoExponent} with the ones in \eqref{Sharp_alpha} by using \eqref{EqExpoentes}.
\smallskip

In this point, we will analyse the following cases:
\smallskip

\begin{itemize}

  \item[({\bf 1})] For instance, in the particular case $m=1$, which is not included in \eqref{CImprov}, we must to highlight that $\alpha_{\mathrm{Hom}} = 1$ and $p_m = 2$ in \eqref{Sharp_alpha}. As a result,
      $$
         \alpha = \frac{(pq-n)r-pq}{q[(p-1)r-(p-2)]},
      $$
      which concurs with Teixeira-Urbano's sharp estimates in \cite[Theorem 3.4]{TU14}.
\smallskip

  \item[({\bf 2})] If $p=2$ then the condition \eqref{CImprov} holds true for any $m>1$. In particular, we recover the recent improved regularity estimates in \cite[Theorem 2.5]{Diehl21} (see also \cite[Theorem 6]{AMU20}).
\smallskip
\smallskip

  \item[({\bf 3})] If $\frac{m-1}{p-1}.\frac{p}{p+m-3} \ge 1$, then \eqref{CImprov} holds trivially. Such a condition implies that
  $$
  m \ge p\left[(p-1)\left(1-\frac{3}{p}\right)+1\right].
  $$
\smallskip

  \item[({\bf $4^{\ast}$})] On the other hand, if $\frac{m-1}{p-1}.\frac{p}{p+m-3} < 1$ in \eqref{CImprov}, then we could compar it with the expected (and conjectured) optimal upper bound in \eqref{SharpAlpha_Hom} (for $m \ge 2$). Hence,
$$
\frac{p-1}{m+p-3} \leq \frac{m-1}{p-1}.\frac{p}{p+m-3}
$$
if and only if
$$
\max\left\{2, \frac{(p-1)^2}{p}+1\right\} \le m < p\left[(p-1)\left(1-\frac{3}{p}\right)+1\right].
$$

\end{itemize}

Therefore, such above scenarios provide necessary conditions in order to our H\"{o}lder regularity estimates improve former ones in a unified fashion. Furthermore, to find the full description of the regions where improved estimates do hold is a non-trivial task, because the explicit representation of $\alpha_{\mathrm{Hom}}$ is an open issue.

\smallskip

In the sequel, we will define the region where improved estimates take place:
\begin{equation}\label{ImprovedRegions}
  \mathcal{I}_{m, p} \defeq \{(m, p)\in (1, \infty)\times [2, \infty): \eqref{CImprov} \quad \text{does hold true}\}.
\end{equation}

We must emphasize that Item (2), (3) and (possibly (4)) ensure that $\mathcal{I}_{m, p} \neq \emptyset$.

\smallskip

Finally, as a consequence of Theorem \eqref{t1.2} we obtain the following result:

\begin{corollary} Let $u$ be a bounded weak solution of \eqref{1.1} in $Q^{-}_1$. Suppose that assumptions of Theorem \eqref{t1.2} are in force. Suppose further that
$$
\frac{1}{r}+\frac{n}{pq}<1 \quad \text{and} \quad \frac{3}{r}+m\left(1-\frac{1}{r}\right)+\frac{n}{q}\le2
$$
there holds. Then, $u \in C_{\text{loc}}^{\alpha, \frac{\alpha}{\theta}}(Q^{-}_1)$ (in the parabolic sense), where $\alpha \in \left(0, \mathrm{M}_{\sharp}\right)$.
\end{corollary}

In conclusion, we highlight that
$$
\mathrm{M}_{\sharp} \to \alpha_{\mathrm{Heat \,\,Operator}}  = 1 \quad \text{as} \quad m\to 1^{+} \quad \text{and} \quad p \to 2^{+}.
$$
This insight suggests that we must examine the stability of our estimates as the parameters of the equation approach the linear case. We will address this statement in a precise quantitative way in Section \ref{SecAsymEst}.

\subsection{Main obstacles and strategies}

It is worth mentioning that different from the linear setting (i.e. $m=1$ and $p=2$), the analysis for the corresponding doubly degenerate one, i.e. $m>1$ and $p>2$, it is more challenging and it involves the development of new ideas and techniques. For this very reason, in our approach, we make use of an $\theta-$intrinsic scaling technique inspired by \cite{AdaSRT}, \cite{DOS18}, \cite{TU14} and \cite{U08}, where $\theta>0$ is the \textit{intrinsic scaling factor} for the temporal variable, which depends on universal parameters of problem and integrability properties of the source term.

In brief, our strategy follows the ideas from \cite{Ara20}, \cite{Diehl21} and \cite{TU14} (see also \cite{AdaSRT} and \cite{JVSilva19}). Nevertheless, we establish a finer and sharper version of \cite[Lemma 3.2]{Ara20}, which allows us to transfer (in a continuous fashion) available regularity estimates of the homogeneous case to the inhomogeneous one, under a smallness regime on the data.

The core idea behind the proof of Theorem \ref{t1.2} is performing a geometric decay argument along those points around which the equation degenerates, i.e. where weak solutions become very small. In effect, the purpose will be to make use of an $(m, p)-$ approximation in a continuous fashion (Lemma \ref{l2.1}), thus ensuring that weak solutions oscillate in a suitable ``geometric'' manner, i.e.
{\small{
$$
C^0-\text{closeness}\quad \stackrel[\Longrightarrow]{\text{Geometric estimate}}{}    \quad \displaystyle\sup_{Q^{-}_{\rho}(x_0, t_0)} \frac{\left|u(x, t)-u(x_0, t_0)\right|}{\rho^{\alpha}}\leq 1,
$$}}
thereby getting the desired H\"{o}lder estimate. Nevertheless, different from $C^{\alpha}$ regularity for degenerate models with continuous coefficients, i.e.
$$
 \frac{\partial u}{\partial t}-\Div(\mathfrak{a}(x, t)\lvert \nabla u \rvert^{p-2}\nabla u) = f(x, t) \quad \text{in} \quad \Omega_T \quad \text{for} \quad p > 2,
$$
we can no longer proceed with an interaction scheme as the one in \cite{TU14}, i.e.
$$
 \displaystyle\sup_{Q^{-}_{\rho^k}(x_0, t_0)} \frac{\left|u(x, t)-\mathfrak{c}_k\right|}{\rho^{k\alpha}}\leq 1 \quad \stackrel[\Longrightarrow]{\text{Dini-Campanato}}{{\small\text{embedding}}} \,\,\,u \,\, \text{is} \,\,\,C^{\alpha, \frac{\alpha}{\theta}} \quad \text{at}\,\,\,(x_0,t_0),
$$
because we do not know, \textit{a priori}, which PDE is fulfilled by
$$
     u_k(x, t) \defeq \frac{u(\rho^k x + x_0, \rho^{k \theta}t + t_0) -  \mathfrak{c}_k}{\rho^{k \alpha}}, \quad \text{for}\quad \{\mathfrak{c}_k\}_{k\in\mathbb{N}}\quad \text{a sequence of constants},
$$
since $u \mapsto  \frac{\partial u}{\partial t}-\Div(m\lvert u\rvert^{m-1}\lvert \nabla u \rvert^{p-2}\nabla u)$ is not translation invariant by constant mappings. Nonetheless, we are able to derive quantitative information on the oscillation of $u$, provided we get a sort of suitable control under the magnitude of the solutions at interior points (see Lemma \ref{induction} for more details).

 Another important aspect of our approach consists of removing the restriction of analyzing the desired $C^{0,\alpha}$ regularity estimates just along the \textit{a priori} unknown set of zero points of solutions (i.e. the set where weak solutions vanishes), where the diffusivity of the equation collapses (cf. \cite{DOS18}, \cite{daSRS18} and \cite{daSS18}, where sharp and improved regularity estimates were obtained along the set of certain degenerate points of existing solutions).

\section{Some useful auxiliary results}

In this section we will present some auxiliary results for our purposes. Next, we define the notion of weak solutions for our problem.

\begin{definition}\label{d2.1} A locally bounded function $u$ is called a local weak solution of \eqref{1.1} in $\Omega\times(0,T]$, if $u \in C_{\loc}(0, T; L^2(\Omega))$ and $\Phi^{\frac{1}{p}}(|u|)|\nabla u| \in L^{p}_{loc}(\Omega_{T})$ for every compact set $K\subset\Omega$, every $[t_1, t_2] \subset (0, T]$ and $\psi\in H^1_{\loc}(0, T; L^2(K))\cap L^p_{\loc}(0,T; W_{\loc}^{1, p}(K))$ there holds
$$
\displaystyle \left.\int_{K} u \psi\,dx \right|^{t_2}_{t_1} + \int_{t_1}^{t_2} \int_{K} \left[-u \frac{\partial \psi}{\partial t} + \mathcal{A}(x, t,u, \nabla u) \cdot \nabla \psi \right]\,dx\,dt = \int_{t_1}^{t_2} \int_{K} f \psi\,dx\,dt.
$$
\end{definition}

We also can present an equivalent definition of weak solutions via Steklov average, which allow us to prove a Caccioppoli type estimate. Indeed, for each $0<h<T$ we denote $u_{h}$ as the Steklov average of $u$ as follows:
\begin{equation}
u_{h}\colon= \left\{
\begin{array}{rcrcl}
  \frac{1}{h}\int\limits_{t}^{t+h} u(.,\tau)d\tau, \ {if} \ t\in (0,T-h],\\
  0 ,\  {if} \ t \in (T-h,T]
\end{array}
\right.
\end{equation}

Before delivering the Caccioppoli type estimate, let us present the following energy estimate:

\begin{proposition}\label{p2.1} Let $u$ a weak solution of \eqref{1.1}. Then, there exists a constant $\gamma=\gamma(\mathrm{C}_{1}, \mathrm{C}_{2}, n, K \times [t_{1},t_{1}],\lvert \lvert f \rvert \rvert_{L^{q, r}(\Omega_{T})})$, such that for each cylinder $\Omega_{T}$ and each level $k$ holds :
{\scriptsize{
\begin{equation*}
    \sup\limits_{t_{1}<t<t_{2}} \int\limits_{K} (u-k)_{\pm}^{2}\xi^pdx +\gamma^{-1} \int\limits_{t_{1}}^{t_{2}}\int\limits_{K} \Phi(|u|)|\nabla(u-k)_{\pm}|^{p}\xi^{p} dxdt \leq
\end{equation*}}}
{\scriptsize{
\begin{equation*}
\gamma \int_{t_{1}}^{t_{2}}\int_{K} (u-k)_{\pm}^{2}\xi^{p-1}|\xi_{t}|dxdt+\gamma\int_{t_{1}}^{t_{2}}\int_{K}\Phi(|u|)(u-k)^{p}_{\pm}|D\xi|^{p} dxdt+ \gamma\left(\int_{t_{1}}^{t}\lvert \mathrm{A}^{\pm}_{k}(t)\rvert^{\frac{r}{q}} dt\right)^{\frac{2p.(1+\kappa)}{r}}
\end{equation*}}}
\end{proposition}

Next Caccioppoli type estimate plays a decisive role in establishing the existence of an $(m, p)$-approximation for weak solutions of \eqref{1.1} under smallness conditions on the data.

\begin{proposition}[{\bf Caccioppoli type estimate}] Let $K \times [t_1, t_2] \subset \Omega \times (0, T]$. If $u$ is a weak solution of \eqref{1.1}, then there exists a constant $\mathrm{C}>0$, depending only on $n$, $p$ and $K\times [t_1, t_2]$ such that
{\tiny{
$$
\begin{array}{rcl}
 \displaystyle \sup\limits_{t \in (t_{1}, t_{2})} \int\limits_{K} u^{2}\xi^p dx + \int\limits_{t_{1}}^{t_{2}}\int\limits_{K} \Phi(|u|)|\nabla u|^{p}\xi^{p} dxdt  & \le &
\displaystyle \mathrm{C} \int\limits_{t_{1}}^{t_{2}}\int\limits_{K} u^{2}\xi^{p-1}\left|\frac{\partial\xi}{\partial {t}}\right|dxdt+\int\limits_{t_{1}}^{t_{2}}\int\limits_{K}\Phi(|u|)|u|^{p}|D\xi|^{p} dxdt \\
   & + & \displaystyle \mathrm{C}\lvert\rvert f \rvert\rvert_{L^{q,r}(\Omega_T)}^{2}.
\end{array}
$$}}
for every $\xi \in C^{\infty}_{0}(K\times (t_{1},t_{2});[0,1])$
\end{proposition}

\begin{proof}
  Take $\varphi=u_{h}\xi^{p}$ as test function and using the equivalent definition of weak solution as in the previous Proposition \cite{p2.1} Then, we can finish the proof by letting $h \to 0$ and using Young's inequality. For a detailed proof we recommend the reader to see \cite[Proposition 3.1]{FS08}.
\end{proof}

In the sequel, we present the available interior regularity for doubly nonlinear degenerate PDEs as in \eqref{1.1}. Such estimates can be found in \cite{Iva89}, \cite{Iva91}, \cite{Iva95}, \cite{Iva97}, \cite{Iva98}, \cite{Iva00}, \cite{PV93} and \cite{Vesp92}.

\begin{theorem}\label{ThmHolderEst} Let $u$ be a locally bounded weak solution of \eqref{1.1}. Suppose that (P1)-(P4) and \eqref{w-cc} there hold. Then, $u$ is locally H\"{o}lder continuous in $\Omega_T$, and for every compact subset $\mathrm{K} \subset \Omega_T$, there exists constants $\gamma>1$ and $\alpha_0 \in (0, 1)$ such that
$$
|u(x_1, t_1)-u(x_2, t_2)| \leq \gamma \cdot \left(|x_1-x_2|^{\alpha_0}+ \|u\|_{L^{\infty}(\mathrm{K})}^{\frac{m-1}{p}}|t_1-t_2|^{\frac{\alpha_0}{p_m}} \right)
$$
for every $(x_1, t_1), (x_2, t_2) \in \mathrm{K}$.
\end{theorem}

Next, let us remember the optimal regularity estimates for evolutionary $p-$Laplacian type equations as follows
\begin{equation}\label{Eqp-Lapla}
  \mathcal{Q}_{p} u \defeq  \frac{\partial u}{\partial t}-\Div(\mathfrak{a}(x, t)\lvert \nabla u \rvert^{p-2}\nabla u) = f(x, t) \quad \text{in} \quad \Omega_T \quad \text{for} \quad p \ge 2,
\end{equation}
whose coefficients are continuous and they fulfill
\begin{equation}\label{ContCondp_laplace}
0 \le \mathrm{L}_0 \le \mathfrak{a}(x, t) \le \mathrm{L}_1< \infty \quad \text{in} \quad \Omega_T.
\end{equation}
Moreover, $f \in L^{q,r}(\Omega_T)$ fulfils the compatibility conditions \eqref{w-cc}. In this direction, we present sharp regularity estimates for the class of equations \eqref{Eqp-Lapla}.

\begin{theorem}\label{ThmSharp-p-Laplace} Let $u$ be a bounded weak solution of \eqref{Eqp-Lapla}. Suppose further that the assumptions \eqref{w-cc} and \eqref{ContCondp_laplace} are in force. Then, $u$ is locally $C^{0, \hat{\alpha}}$ in space variable and $C^{0, \frac{\hat{\alpha}}{\hat{\theta}}}$ in time variable,  where
$$
\hat{\alpha} \defeq \frac{(pq-n)r-pq}{q[(p-1)r-(p-2)]} \qquad \text{and} \qquad \hat{\theta} \defeq 2\hat{\alpha} + (1-\hat{\alpha})p.
$$
\end{theorem}

The proof of Theorem \ref{ThmSharp-p-Laplace} was addressed in \cite[Theorem 3.4]{TU14} for equations with frozen coefficients. Furthermore, a similar result for the setting with continuous variable coefficients as in \eqref{Eqp-Lapla} (under the assumption \eqref{ContCondp_laplace}) can be obtained, see e.g. \cite[Sec. 4]{TU14}.

In the sequel, we present a key step which allow us to access the tangential path toward the regularity theory available for ``frozen'' coefficient, homogeneous $(m, p)-$caloric functions. As a matter of fact, such a result is a compactness devise that states that if the source term $f$ and $\Theta_{\mathcal{A}}$ have respectively a small norm in $L^{q, r}$ and oscillates in a small fashion, then a weak solution to \eqref{1.1} is close to an $(m, p)-$Laplacian profile in an inner sub-domain. The proof is based on a \textit{Reductio ad absurdum} argument and makes use of compactness driven from a Caccioppoli type estimate.

\begin{lemma}[{\bf $(m, p)-$Approximation}]\label{l2.1}
If $u$ is a weak solution of \eqref{1.1} in $Q^{-}_1$ with
$\|u\|_{L^{\infty}(Q^{-}_1)} \leq 1$, then $\forall\varepsilon>0$ there exists $\delta= \delta(p,n,m,\varepsilon)>0$ such that whenever
$$
\displaystyle \max\left\{\|f\|_{L^{q, r}(Q^{-}_1)},  \sup_{0< \rho \leq \rho_0} \sup_{(y, \tau),(x,t) \in Q_{\rho}} \Theta_{\mathcal{A}}(y, \tau, x, t)\right\}\leq\delta_{\varepsilon}
$$
there exists an $(m, p)$-caloric function $\phi: Q^{-}_{\frac{1}{2}} \to \R$, i.e.,
\begin{equation}\label{EqHomProb}
      \frac{\partial \phi}{\partial t}-\Div(\mathcal{A}(\phi, \nabla \phi)) = 0 \quad \text{in} \quad Q^{-}_{\frac{1}{2}},
\end{equation}
with $\mathcal{A}$ satisfying (P1)-(P3) with $\mathrm{C}_1 = \mathrm{C}_2, \gamma_1=\gamma_2, \psi_2=\psi_2$ and $\omega \equiv 0$ such that
\begin{equation}\label{3.1}
\|u-\phi\|_{L^{\infty}\left(Q^{-}_{\frac{1}{2}}\right)} <\varepsilon.
\end{equation}
\end{lemma}

\begin{proof} Suppose for the sake contradiction that the statement of the Lemma does not hold. Then, for some $\delta_{0}>0$, there would exist sequences
$$
(u_{j})_{j} \in C^0_{loc}(0,-1;L^{2}_{loc}(\Omega)), \quad \text{with} \quad  \lvert u_{j} \rvert \Phi(|u_{j}|)^{\frac{1}{p}} \in L^{p}_{loc}(0,-1;W^{1,p}(B_{1}))
$$
and $(\mathcal{A}_j)_j$ and $(f_{j})_{j} \in L^{q,r}(Q^{-}_{1})$ satisfying (P1)-(P4), and linked through
\begin{equation}\label{Eq2.9}
    \frac{\partial u_{j}}{\partial t}-\Div{\mathcal{A}_j(x,t,u_{j}, \nabla u_{j})} = f_{j}(x, t) \quad \text{in} \quad Q^{-}_{1}
\end{equation}
 with
\begin{equation}\label{Eq2.10}
\lvert\lvert u_{j}\rvert\rvert_{L^{\infty}(Q^{-}_{1})} \leq 1 \qquad \text{and} \qquad \max\left\{\lvert \lvert f_{j} \rvert\rvert_{L^{q,r}(Q^{-}_{1})},  \,\,\Theta_{\mathcal{A}_j}(y, \tau, x, t) \right\} \leq\frac{1}{j}
\end{equation}
for all $j$. However, for any weak solution $\phi$ of \eqref{EqHomProb} we have
\begin{equation}\label{Eq2.11}
\lvert\lvert u_{j}-\phi \rvert \rvert_{L^{\infty}\left(Q^{-}_{\frac{1}{2}}\right)} >\delta_{0} \quad \text{for all} \quad j.
\end{equation}

Now, consider a cutoff function $\xi \in C^{\infty}_{0}(Q^{-}_{1})$, such that $0\le \xi \le 1$, with $\xi = 1$ in $Q_{\frac{1}{2}}$ and $\xi = 0$ in $\partial_{p} Q^{-}_{1}$. Thus, since $u_{j}$ is a solution of equation \eqref{Eq2.9}, the Cacciopolli estimate (Proposition \ref{p2.1}) ensures, with the help of \eqref{Eq2.10}, that
\begin{equation*}
    \sup\limits_{-1<t<0} \int\limits_{B_{1}} u_j^{2}\xi^p dx + \int\limits_{-1}^{0}\int\limits_{B_{1}} \Phi(|u_j|)|\nabla u_j|^{p}\xi^{p} dxdt \leq \bar{\mathrm{C}}.
\end{equation*}

Now, let $v_{j}=\mathrm{F}(|u_{j}|)$, where $\mathrm{F}$ is a primitive of $\Phi^{\frac{1}{p}}$, then
$$
|\nabla v_{j}|^{p}=\Phi(|u_{j}|)|\nabla u_{j}|^{p}.
$$

Hence,
$$
\lvert\rvert \nabla v_{j} \lvert\lvert_{L^p \left(Q^{-}_{\frac{1}{2}}\right)}^{p} = \int\limits_{-2^{\theta}}^{0}\int\limits_{B_{1}}\lvert \nabla v_{j}\rvert \rvert^{p}dxdt \leq \int\limits_{-1}^{0}\int\limits_{B_{1}}\Phi(|u_{j}|)\lvert \nabla u_{j}\rvert \rvert^{p}\xi^{p}dxdt \leq \bar{\mathrm{C}}
$$

Therefore, up to subsequence,
\begin{equation}\label{Eq2.12}
\nabla v_{j} \rightharpoonup \Xi \quad \text{weakly in} \quad L^{p}\left(Q^{-}_{\frac{1}{2}}\right).
\end{equation}

Moreover, from Theorem \ref{ThmHolderEst}, the sequence $(u_{j})_j$ is equi-continuous, then by Arzel\`{a}-Ascoli compactness criterium, up to a subsequence,
\begin{equation}\label{Eq2.13}
u_{j} \to u_{\infty} \qquad \text{uniformly in}  \qquad Q^{-}_{\frac{1}{2}}.
\end{equation}
In this point, we can to identity by means of \eqref{Eq2.12} and \eqref{Eq2.13} and the uniqueness of weak limits $\Xi= \nabla v$, since we have the point-wise convergence
$$
v_{j}=\mathrm{F}(|u_{j}|) \to \mathrm{F}(u_{\infty}) \defeq v.
$$

Finally, by passing the limit in \eqref{Eq2.9}, we conclude that $v$ satisfies \eqref{EqHomProb}, which yields a contradiction with \eqref{Eq2.11} for $j \gg 1$. This completes the proof.

\end{proof}

\begin{remark}[{\bf Normalization and ``smallness regime''}]\label{r2.1}
Note that if $u$ is any weak solution of \eqref{1.1} in $Q_1$, then it is possible to normalize it in such a way, that the normalized function satisfies conditions of Lemma \ref{l2.1}. More precisely, for a $\delta>0$ and $s>0$ fixed, there exists positive constant $\mu=\mu(\delta, s, \|u\|_{L^{\infty}}, \|f\|_{L^{q, r}})$ such that the function
$$
v_{\mu_0}(x,t)\defeq \mu_0^{s}u(\mu_0^{s}x,\mu_0^{\tau}t),
$$
fulfils the assumption of Lemma \ref{l2.1}, where $\tau:=s(m-1)+2s(p-1)>0$,
$$
0<\mu_0 \le \min\left\{ 1, \frac{1}{ \sqrt[s]{\|u\|_{L^{\infty}(Q^{-}_1)}}},\, \sqrt[\kappa_0]{\frac{\delta}{\|f\|_{L^{q, r}(Q^{-}_1)}+1}},\sqrt[\Pi_0]{\frac{\delta}{\omega^{-1}_{\mathcal{A}}(\frac{\delta}{\mathrm{C_{\mathcal{A}}}+1})}} \right\},
$$
$$
\begin{array}{rcl}
  \kappa_0 & \defeq & s[2(p-1)+m]-\left(\frac{sn}{q}+\frac{(2p-1)s+s(m-1)}{r}\right) \\
   & = & s(m+p-2)\left[\left(1-\frac{1}{r}\right)+\frac{1}{r(m+p-2)}\right] + sp\left[1-\left(\frac{n}{pq}+\frac{1}{r}\right)\right]
\end{array}
$$
is a positive constant due to integrability condition from \eqref{w-cc} and
$$
\Pi_0 \defeq s(m-1)+2s(p-1)=s[2p+m-3]>0, \quad \text{since} \quad m\ge1 \quad \text{and} \quad p\geq 2.
$$
\end{remark}

\section{Sharp geometric  estimates via an iterative scheme}\label{Sec3}

Using an iterative scheme, we will prove the desired regularity estimate.
The estimate in Lemma \ref{l2.1} can be further improved up to the sharp exponent in our compatibility regime \eqref{w-cc} providing a precise control of oscillation at interior points with controlled magnitude. The following lemma serves such a purpose and it provides the first step of such an iteration.

\begin{lemma}\label{1stStepInduc} There exist $\varepsilon > 0$ and $\lambda \in \left(0,  \frac{1}{4}\right]$ both depending only universal parameters $m,n,p$ and $\alpha$, such that if
$$
\max\left\{\lvert\lvert f\rvert \rvert_{L^{q,r}(Q_{1})}, \,\, \Theta_{\mathcal{A}}(y, \tau, x, t)\right\}<\varepsilon
$$
and $u$ is a weak solution of \eqref{1.1} in $Q^{-}_{1}$, with $\lvert\lvert u\rvert\rvert_{L^{\infty}\left(Q^{-}_{1}\right)}\leq 1$,
then
$$
\lvert\lvert u\rvert\rvert_{L^{\infty}\left(Q^{-}_{\lambda}\right)}\leq \lambda^{\alpha} \qquad  \text{provided} \qquad |u(0,0)|\leq \frac{\lambda^{\alpha}}{4}.
$$
\end{lemma}

\begin{proof} Let us fix a $\delta \in (0, 1)$ to be chosen in a precise way \textit{a posteriori}. Next, we apply $(m, p)-$Approximation Lemma \ref{l2.1}, which provides an $\varepsilon>0$ small enough and a weak solution $\phi$ of \eqref{EqHomProb} such that
$$
\lvert\lvert u-\phi \rvert\rvert_{L^{\infty}\left(Q^{-}_{\frac{1}{2}}\right)}\leq\delta
$$

 It follows from \cite{Iva95} and \cite{PV93} that $\phi \in C_{\text{loc}}^{\alpha_{\mathrm{Hom}}, \frac{\alpha_{\mathrm{Hom}}}{p_m}}(Q^{-}_1)$ for a H\"{o}lder exponent $0<\alpha_{\mathrm{Hom}}\leq 1$ and $p_m$ as in \eqref{p_m}. Particularly, in the $\theta$-parabolic cylinder we have
{\scriptsize{
$$
   |\phi(x,t)-\phi(y,s)|\leq \gamma\left( |x-y|^{\alpha_{\mathrm{Hom}}} + \|u\|_{L^{\infty}\left(Q^{-}_{\frac{1}{2}}\right)}^{\frac{m-1}{p}}\sqrt[p_m]{|t-s|^{\alpha_{\mathrm{Hom}}}}\right) \quad \forall\,\,(x,t),(y,s) \in Q^{-}_{\frac{1}{2}}.
$$}}
Now, notice that by choosing $\lambda \in \left(0,  \frac{1}{4}\right]$ this yields $Q^{-}_{\lambda}\subset B_{\frac{1}{2}}\times \left(- \frac{1}{2^{\theta}},0\right]=Q^{-}_{\frac{1}{2}}$.

Hence, for $m>1$ (and \eqref{ImprovedRegions} in force) we claim that $\phi$ satisfies
$$
\sup\limits_{(x,t)\in G_{\lambda}}|\phi(x,t)-\phi(0,0)|\leq \mathrm{C}\lambda^{\frac{\theta\alpha_{\mathrm{Hom}}}{p}} \quad (\text{resp.} \,\,\, \cdots \le \mathrm{C}\lambda \,\,\,\text{if}\,\,\,m=1)
$$
for $\lambda \ll 1$, to be chosen, and $\mathrm{C}>1$. On the other hand, we obtain for any $m\ge 1$ and $p \ge 2$ the following
$$
\sup\limits_{(x,t)\in G_{\lambda}}|\phi(x,t)-\phi(0,0)|\leq \mathrm{C}\lambda^{\frac{2\alpha_{\mathrm{Hom}}(p-1)}{p_m(p+m-2)}}.
$$

Indeed, for $(x,t) \in Q^{-}_{\lambda}$ and $m>1$ (and \eqref{ImprovedRegions} in force) we get
 $$
 \begin{array}{rcl}
   \lvert \phi(x,t) -\phi(0,0) \rvert & \le & \lvert \phi(x,t)-\phi(0,t)\rvert + \lvert \phi(0,t)-\phi(0,0)\rvert \\
    & \le & \mathrm{k}_{1}\lvert x - 0\rvert^{\alpha_{\mathrm{Hom}}}+\mathrm{k}_{2}\lvert t-0 \rvert^{\frac{\alpha_{\mathrm{Hom}}}{p}} \\
    & \le & \mathrm{k}_{1}\lambda^{\alpha_{\mathrm{Hom}}}+\mathrm{k}_{2}\lambda^{\frac{\theta\alpha_{\mathrm{Hom}}}{p}} \\
    & \le & \max\{\mathrm{k}_1, \mathrm{k}_2\}\lambda^{\frac{\theta\alpha_{\mathrm{Hom}}}{p}}, \quad (\text{resp.} \,\,\,\le \max\{\mathrm{k}_1, \mathrm{k}_2\}\lambda \,\,\,\text{if}\,\,\,m=1),
 \end{array}
 $$
where we have used for $m>1$ that $\theta \le p$, which implies $\alpha_{\mathrm{Hom}}\le \frac{\theta\alpha_{\mathrm{Hom}}}{p}$ (resp. for $m=1$, it holds that $\alpha_{\mathrm{Hom}} = 1$ and $2\le \theta\leq p$, which implies $1\leq \frac{\theta}{2})$.

On the other hand, for any $m\ge 1$ and $p \ge 2$ since
$$
\theta \geq 1+\frac{p-1}{p+m-2}\geq \frac{2(p-1)}{p+m-2}
$$
we get
$$
 \begin{array}{rcl}
   \lvert \phi(x,t) -\phi(0,0) \rvert & \le & \lvert \phi(x,t)-\phi(0,t)\rvert + \lvert \phi(0,t)-\phi(0,0)\rvert \\
    & \le & \mathrm{k}_{1}\lvert x - 0\rvert^{\alpha_{\mathrm{Hom}}}+\mathrm{k}_{2}\lvert t-0 \rvert^{\frac{\alpha_{\mathrm{Hom}}}{p}} \\
    & \le & \mathrm{k}_{1}\lambda^{\alpha_{\mathrm{Hom}}}+\mathrm{k}_{2}\lambda^{\frac{\theta\alpha_{\mathrm{Hom}}}{p}} \\
    & \le & \mathrm{k}_{1}\lambda^{\frac{\alpha_{\mathrm{Hom}}(p-1)}{p+m-2}} + \mathrm{k}_{2}\lambda^{\frac{2\alpha_{\mathrm{Hom}}(p-1)}{p_m(p+m-2)}}\\
    & \le & \max\{\mathrm{k}_1, \mathrm{k}_2\}\lambda^{\frac{2\alpha_{\mathrm{Hom}}(p-1)}{p_m(p+m-2)}},
 \end{array}
 $$

Therefore, for $m>1$ (and \eqref{ImprovedRegions} in force) we can to estimate by using of $\delta$-approximation (Lemma \ref{l2.1}) the following
$$
\begin{array}{rcl}
  \sup\limits_{Q^{-}_{\lambda}}\lvert u \rvert & \le & \sup\limits_{Q^{-}_{\frac{1}{2}}} \vert u-\phi \rvert + \sup\limits_{Q^{-}_{\lambda}}\lvert \phi(x,t)-\phi(0,0)\rvert + \lvert u(0,0)-\phi(0,0)\rvert + \lvert u(0,0) \rvert\\
   & \le & 2\delta + \mathrm{C}\lambda^{\frac{\theta\alpha_{\mathrm{Hom}}}{p}}+\frac{\lambda^{\alpha}}{4} \quad (\text{resp.} \,\,\,\le 2\delta + \mathrm{C}\lambda+\frac{\lambda^{\alpha}}{4} \,\,\,\text{if}\,\,\,m=1).
\end{array}
$$

On the other hand, for any $m\ge 1$ and $p \ge 2$ we have
$$
\begin{array}{rcl}
  \sup\limits_{Q^{-}_{\lambda}}\lvert u \rvert & \le & \sup\limits_{Q^{-}_{\frac{1}{2}}} \vert u-\phi \rvert + \sup\limits_{Q^{-}_{\lambda}}\lvert \phi(x,t)-\phi(0,0)\rvert + \lvert u(0,0)-\phi(0,0)\rvert + \lvert u(0,0) \rvert\\
   & \le & 2\delta + \mathrm{C}\lambda^{\frac{2\alpha_{\mathrm{Hom}}(p-1)}{p_m(m+p-2)}}+\frac{\lambda^{\alpha}}{4}.
\end{array}
$$

Finally, by taking
$$
  \lambda \in \left(0, \,\min\left\{\frac{1}{4}, \left( \frac{1}{4\mathrm{C}}\right)^{\frac{p}{\theta\alpha_{\mathrm{Hom}}-p\alpha}}\right\}\right] \qquad \text{and} \qquad \delta \in \left(0, \frac{\lambda^{\alpha}}{4}\right] \quad \text{for} \quad m>1,
$$
(and \eqref{ImprovedRegions} in force) and
$$
\text{resp.} \quad \lambda \in \left(0, \,\min\left\{\frac{1}{4}, \left( \frac{1}{4\mathrm{C}}\right)^{\frac{1}{1-\alpha}}\right\}\right] \qquad \text{and} \qquad \delta \in \left(0, \frac{\lambda^{\alpha}}{4}\right] \quad \text{for} \quad m=1,
$$
On the other hand, for any $m\ge 1$ and $p \ge 2$ we get
$$
  \lambda \in \left(0, \,\min\left\{\frac{1}{4}, \left( \frac{1}{4\mathrm{C}}\right)^{\frac{p_m(m+p-2)}{2\alpha_{\mathrm{Hom}}(p-1)-\alpha p_m(m+p-2)}}\right\}\right] \quad \text{and} \quad \delta \in \left(0, \frac{\lambda^{\alpha}}{4}\right],
$$
and plugging all of them in the previous inequality we obtain the desired estimate.
\end{proof}

\begin{remark} It is important to stress that in the previous Lemma we must impose
$$
 \theta\alpha_{\mathrm{Hom}}-p\alpha>0 \qquad (\text{resp}. \quad 2\alpha_{\mathrm{Hom}}(p-1)-\alpha p_m(m+p-2)>0).
$$
Moreover, by invoking the definition of $\theta$ in \eqref{Int-Scal}, we get for $m>1$ (and \eqref{ImprovedRegions} in force)
$$
\alpha \in \left(0, \, \frac{p\alpha_{\mathrm{Hom}}}{p+\alpha_{\mathrm{Hom}}(m+p-3)}\right) \quad (\text{resp.}\,\,\,\alpha \in (0, \,1)\,\,\,\text{for} \,\,\,m=1).
$$
On the other hand, for any $m\ge 1$ and $p \ge 2$ we have
$$
\alpha \in \left(0, \, \frac{2\alpha_{\mathrm{Hom}}(p-1)}{p_m(m+p-2)}\right).
$$

In any case, since $\alpha_{\mathrm{Hom}} = 1$ and $p_m=2$ when $m=1$, we can re-write such conditions in a unified way
$$
\alpha \in \left(0, \, \max\left\{\frac{p\alpha_{\mathrm{Hom}}}{p_m+\alpha_{\mathrm{Hom}}(m+p-3)}, \frac{2\alpha_{\mathrm{Hom}}(p-1)}{p_m(m+p-2)}\right\}\right),
$$
where $p_m$ is given by \eqref{p_m}.
\end{remark}

In order to obtain a precise control on the influence of magnitude of $|u(0, 0)|$, we will iterate solutions
(using  Lemma \ref{1stStepInduc}) in parabolic $\lambda-$adic cylinders.

\begin{lemma}\label{induction} Suppose that the assumption of Lemma \ref{1stStepInduc} are in force. Then,
\begin{equation}\label{Induc-k-Step}
  \lvert\lvert u\rvert\rvert_{L^{\infty}\left(Q^{-}_{\lambda^{k}}\right)}\leq \lambda^{\alpha k} \qquad \text{provided} \qquad  |u(0,0)|\leq \frac{\lambda^{\alpha k}}{4} \qquad \text{for all} \quad k \in \mathbb{N}.
\end{equation}
\end{lemma}

\begin{proof} The proof follows by induction process. The case $k=1$ is precisely the statement of Lemma \ref{1stStepInduc}. Suppose now that \eqref{Induc-k-Step} holds for all the values of $j = 1, 2, \cdots, k$. Our goal is to check it for $j = k+1$. For this purpose, define $v_k: Q^{-}_{1}\to \mathbb{R}$ given by
$$
v_k(x,t) \defeq \frac{u(\lambda^{k}x,\lambda^{k\theta}t)}{\lambda^{k\alpha}}.
$$

The function $v_k$ is a weak solution of
$$
\frac{\partial v_k}{\partial t}-\Div(\mathcal{A}_k(x,t,v_k,\nabla v_k))=f_k(x,t) \quad \text{in}  \quad  Q^{-}_{1},
$$
where
$$
\left\{
\begin{array}{rcl}
  \mathcal{A}_k(x,t, s, \xi) & \defeq  & \lambda^{-k\alpha(m+p-2)+k(p-1)}\mathcal{A}(\lambda^{k}x,\lambda^{k\theta}t,\lambda^{k\alpha}s,\lambda^{(\alpha-1)k}\xi) \\
  f_k(x, t) & \defeq & \lambda^{-k\alpha(m+p-2)+k(p-1)+k}f(\lambda^{k}x,\lambda^{k\theta}t).
\end{array}
\right.
$$

It is straightforward to check that $\mathcal{A}_k$ fulfills the properties (P1)-(P2). Moreover, the induction hypothesis implies that
$$
  \lvert\lvert v_k \rvert\rvert_{L^{\infty}\left(Q^{-}_{1}\right)} \leq 1 \quad \text{and} \quad |v_k(0,0)|=\frac{|u(0,0)|}{\lambda^{k\alpha}}\leq \frac{\lambda^{(k+1)\alpha}}{4\lambda^{k\alpha}}=\frac{\lambda^{\alpha}}{4}.
$$

Moreover, we can estimate
$$
\begin{array}{rcl}
  \lvert \lvert f_k \rvert\rvert_{L^{q,r}(Q_{1})}^{r} & = & \displaystyle  \int\limits_{-1}^{0}\Bigg(\int\limits_{B_{1}}\lvert f_k(x,t)\rvert^{q} dx\Bigg)^{\frac{r}{q}}dt\\
   & = & \displaystyle \int\limits_{-1}^{0}\Bigg(\int\limits_{B_{1}}\lambda^{\left(-k\alpha(m+p-2)+k(p-1)+k\right)q}\lvert f(\lambda^{k}x,\lambda^{k\theta} t)\rvert^{q}dx\Bigg)^{\frac{r}{q}}dt\\
   & = & \displaystyle \int\limits_{-1}^{0}\Bigg(\int\limits_{B_{\lambda^{k}}}\lambda^{\left[\left(-k\alpha(m+p-2)+k(p-1)+k\right)q-nq\right]}\lvert f(z,\lambda^{k\theta} t)\rvert^{q}dz\Bigg)^{\frac{r}{q}} dt\\
   & = & \displaystyle \lambda^{\left[\left(-k\alpha(m+p-2)+k(p-1)+k\right)q-nq\right]\frac{r}{q}}.\lambda^{-k\theta}\int\limits_{-\lambda^{k\theta}}^{0}\Bigg(\int\limits_{B_{\lambda^{k}}}|f(z,\tau)|^{q} dz\Bigg)^{\frac{r}{q}}d\tau
\end{array}
$$

Hence,
$$
\lvert\lvert f_k \rvert\rvert_{L^{q,r}(Q^{-}_{1})}^{r} \leq \lambda^{\left[\left(-k\alpha(m+p-2)+k(p-1)+k\right)q-nk\right]\frac{r}{q}-k\theta}\lvert\lvert f \rvert\rvert_{L^{q,r}(Q^{-}_{\lambda^{k}})}^{r}.
$$
Now, observe that
{\scriptsize{
$$
\left[\left(-k\alpha(m+p-2)+k(p-1)+k\right)q-nk\right]\frac{r}{q}-k\theta \geq 0 \iff \alpha \leq \frac{r(pq-n)-pq}{q[(m+p-2)r-(m+p-3)]}
$$}}

As a result, since $\lambda \in \left(0, \frac{1}{4}\right]$, it follows that
$$
\lvert\lvert f_k \rvert\rvert_{L^{q,r}(Q^{-}_{1})}\leq \lvert\lvert f \rvert\rvert_{L^{q,r}(Q^{-}_{\lambda^{k}})} \leq \lvert\lvert f \rvert\rvert_{L^{q,r}(Q^{-}_{1})}\leq \delta.
$$
Finally,
$$
\sup_{0< \rho \leq \rho_{0}} \sup_{(y, \tau),(x,t) \in Q_{\rho}} \Theta_{\mathcal{A}_k}(y, \tau, x, t) \leq \sup_{0< \rho \leq \rho_{0}} \sup_{(y, \tau),(x,t)  \in Q_{\rho^{k}}} \Theta_{\mathcal{A}}(y, \tau, x, t) \leq \delta
$$

Therefore, $v_k$ falls into the hypothesis of Lemma \ref{1stStepInduc}. Hence, we conclude that
$$
  \lvert \lvert v_k \rvert\rvert_{L^{\infty}\left(Q^{-}_{\lambda}\right)} \leq \lambda^{\alpha} \qquad \Longrightarrow \qquad \lvert\lvert u\rvert\rvert_{L^{\infty}\left(Q^{-}_{\lambda^{(k+1)}}\right)}\leq \lambda^{\alpha (k+1)},
$$
thereby concluding the induction process.
\end{proof}

\section{Proof of Theorem \ref{t1.2}}

Before proving our main Theorem, next result ensures that the small magnitude control is not restrictive and works in the case of small radii.

\begin{proposition}\label{rho} Let $u$ be a bounded weak solution of \eqref{1.1} in $Q^{-}_{1}$ and $\lambda$ as in Lemma \ref{1stStepInduc}. Then, for $\rho \in (0, \lambda)$ and $\mathrm{C}>0$ a universal constant, we have
$$
\lvert \lvert u \rvert \rvert_{L^{\infty}\left(Q^{-}_{\rho}\right)}\leq \mathrm{C}\rho^{\alpha} \quad  \text{provided} \quad |u(0,0)|\leq \frac{\rho^{\alpha}}{4}.
$$
\end{proposition}

\begin{proof} Let us define the auxiliary function $\zeta: Q^{-}_{1} \to \mathbb{R}$ given by
$$
\zeta(x,t) \defeq \xi u(\xi^{\mathrm{a}}x,\xi^{m+p-3+p\mathrm{a}}t)
$$
with constants $\mathrm{a}>0$ and $\xi>0$ to be determined \textit{a posteriori}.

Notice that $\zeta$ satisfies in weak sense
$$
\frac{\partial \zeta}{\partial t}-\Div(\hat{\mathcal{A}}(x,t,\zeta,\nabla \zeta))=\hat{f}(x,t) \qquad \text{in} \qquad Q^{-}_{1},
$$
where
$$
\left\{
\begin{array}{rcl}
  \hat{\mathcal{A}}(x,t,s, \varsigma) & \defeq & \xi^{(1+\mathrm{a})(p-1)+(m-1)}\mathcal{A}(\xi^{\mathrm{a}}x,\xi^{m+p-3+p\mathrm{a}}t,\xi^{-1}s,\xi^{-(1+\mathrm{a})}\varsigma) \\
  \hat{f}(x,t) & \defeq & \xi^{(1+\mathrm{a})(p-1)+(m-1)+\mathrm{a}}f(\xi^{\mathrm{a}}x,\xi^{m+p-3+p\mathrm{a}}t)
\end{array}
\right.
$$
Now,
$$
\lvert\lvert \hat{f} \rvert\rvert_{L^{q,r}(Q^{-}_1)}^{r} \leq \xi^{[((p-1)(1+\mathrm{a})+m-1+\mathrm{a})q-n\mathrm{a}]\frac{r}{q}-(m+p-3+p\mathrm{a})}\lvert\lvert f \rvert\rvert_{L^{q,r}(Q_{1})}^{r},
$$
where $\mathrm{a}>0$ is chosen is such way that
$$
[((p-1)(1+\mathrm{a})+m-1+a)q-n\mathrm{a}]\frac{r}{q}-(m+p-3+p\mathrm{a})>0,
$$
which it is possible due to the compatibility conditions \eqref{w-cc} by doing
$$
\mathrm{a}<\frac{q[(m+p-2)(r-1)+1]}{pq(r-1)-nr}.
$$

Moreover, taking $\xi \ll 1$, we fall into the smallness regime required in Lemma \ref{Induc-k-Step}, i.e.,
$$
\lvert\lvert \zeta \rvert\rvert_{L^{\infty}(Q^{-}_{1})} \leq 1 \quad  \text{and} \quad \max\left\{\lvert\lvert \hat{f} \rvert\rvert_{L^{q,r}(Q^{-}_1)}, \,\,\Theta_{\hat{\mathcal{A}}}(y, \tau, x, t)\right\} \leq\varepsilon.
$$

Therefore, given $\rho \in (0, \lambda)$, there exists $k \in \mathbb{N}$ such that
$$
\lambda^{k+1}<\rho \leq \lambda^{k}.
$$
Moreover, since
$$
|u(0,0)|\leq \frac{\rho^{\alpha}}{4}\leq \frac{\lambda^{k\alpha}}{4}
$$
it follows from Lemma \ref{Induc-k-Step} that
$$
   \lvert\lvert u \rvert \rvert_{L^{\infty}\left(Q^{-}_{\lambda^{k}}\right)} \leq \lambda^{k\alpha}.
$$

In conclusion, we obtain the following estimate
$$
\lvert\lvert u \rvert \rvert_{L^{\infty}(Q^{-}_{\rho})}\leq \lvert\lvert u \rvert \rvert_{L^{\infty}\left(Q^{-}_{\lambda^{k}}\right)} \leq \lambda^{k\alpha} \leq \Big(\frac{\rho}{\lambda} \Big)^{\alpha} = \mathrm{C}\rho^{\alpha}.
$$
\end{proof}

Finally, we are in a position to prove Theorem \ref{t1.2}.

\begin{proof}[{\bf Proof of Theorem \ref{t1.2}}] We claim that it is enough to prove that there exists a constant $\mathrm{M}_0(\verb"universal")>0$ such that
\begin{equation}\label{MainEstThm}
  \lvert \lvert u-u(0,0) \rvert \rvert_{L^{\infty}\left(Q^{-}_{\rho}\right)} \leq \mathrm{M}_0\rho^{\alpha}.
\end{equation}

 In effect, since $u$ is a continuous function, we can define
 \begin{equation}\label{Def-mu}
   \mu \defeq (4|u(0,0)|)^{1/\alpha}\geq 0.
 \end{equation}
Now, let us consider $\rho \in (0, \lambda)$. Thus, the analysis will follow by considering three independent cases:
\begin{itemize}
  \item[({\bf 1})] If $ \rho \in [\mu, \lambda)$:

  In this case we have
 $$
   |u(0,0)|= \frac{\mu^{\alpha}}{4} \leq \frac{\rho^{\alpha}}{4}.
 $$
  Therefore, Proposition \ref{rho} yields
\begin{equation}\label{FirstEstThm}
  \displaystyle  \sup\limits_{Q^{-}_{\rho}} |u(x,t)-u(0,0)|\leq \mathrm{C}\rho^{\beta} + |u(0,0)| \leq \Big(\mathrm{C}+\frac{1}{4} \Big)\rho^{\alpha}.
\end{equation}

  \item[({\bf 2})] If $\rho \in (0, \mu)$:

  For $(x,t) \in Q^{-}_{1}$, we define
$$
w(x,t) \defeq \frac{u(\mu x, \mu^{\theta}t)}{\mu^{\alpha}}.
$$

Notice that $|w(0, 0)| = \frac{1}{4}$. Moreover, $w$ fulfils in the weak sense
$$
\frac{\partial w}{\partial t}-\Div(\mathcal{A}_{\mu}(x,t,w,\nabla w))= f_{\mu}(x,t) \qquad \text{in} \qquad Q^{-}_{1},
$$
where
$$
\left\{
\begin{array}{rcl}
  \mathcal{A}_{\mu}(x,t, s, \varsigma) & \defeq & \mu^{-[\alpha(m-1)+(\alpha-1)(p-1)]}\mathcal{A}(\mu x, \mu^{\theta} t, \mu^{\alpha}s, \mu^{\alpha -1}\varsigma) \\
  f_{\mu}(x, t) & \defeq & \mu^{-[\alpha(m-1)+(\alpha-1)(p-1)]+1}f(\mu x, \mu^{\theta} t).
\end{array}
\right.
$$

Therefore, one more time Proposition \ref{rho} applied to $u$, it follows that
\begin{equation}
\lvert \lvert w \rvert \rvert_{L^{\infty}(Q^{-}_{1})} = \frac{1}{\mu^{\alpha}}\lvert \lvert u \rvert \rvert_{L^{\infty}(Q^{-}_{\mu})}\leq \frac{\mathrm{C}\mu^{\alpha}}{\mu^{\alpha}}= \mathrm{C},
\end{equation}
since $|u(0,0)|=\frac{\mu^{\alpha}}{4}$. Such a uniform estimate, together with $C_{\text{loc}}^{0,\alpha}$-regularity estimate from Theorem \ref{ThmHolderEst}, ensure us the existence of a radius $\rho_{0}>0$, depending on the data, such that
$$
\lvert w(x,t) \rvert \geq \frac{1}{8} \quad \text{for all} \quad (x,t) \in Q^{-}_{\rho_{0}}.
$$
In effect, from $\alpha_0-$H\"{o}lder estimates (Theorem \ref{ThmHolderEst}) we obtain in $Q^{-}_{\rho_{0}}$
$$
\begin{array}{rcl}
  \frac{1}{4} & = & |w(0, 0)| \\
   & \le & |w(x_m, t_m)-w(0, 0)|+|w(x_m, t_m)| \\
   & \le & \gamma \cdot \left(|x_m|^{\alpha_0}+ |t_m|^{\frac{\alpha_0}{p_m}} \right) +|w(x_m, t_m)|\\
   & \le & \gamma \cdot \left(\rho_0^{\alpha_0}+ \rho_0^{\frac{\alpha_0\theta}{p_m}} \right) + \frac{1}{8} \\
   & \le & 2\gamma\rho_0^{\frac{\alpha_0\theta}{p_m}} + \frac{1}{8},
\end{array}
$$
where $(x_m, t_m) \in Q^{-}_{\rho_{0}}$ is a point which $w$ achieves the minimum. Hence,
\begin{equation}\label{Estrho_0}
  \rho_0 \ge \left(\frac{1}{16\gamma}\right)^{\frac{p_m}{\alpha_0\theta}}.
\end{equation}

As a result, $w$ satisfies in the weak sense
$$
\frac{\partial w}{\partial t}-\Div(\mathrm{a}_0(x,t)|\nabla w|^{p-2}\nabla w))= f(x, t) \quad \text{in} \quad  Q^{-}_{\rho_{0}},
$$
where $f \in L^{q, r}(Q^{-}_1)$ (with \eqref{w-cc} in force) and $(x, t) \mapsto \mathfrak{a}_0(x,t) = \mathfrak{a}_0(x,t, w)$ is a continuous function. Moreover,
$$
0<\mathrm{C}_1\left(\frac{1}{8}\right)^{m-1}\le \mathfrak{a}_0(x,t, w) \leq \mathrm{C}_2 \lvert \lvert w \rvert \rvert_{L^{\infty}(Q^{-}_{1})}^{m-1}< \infty
$$
for positive constants $\mathrm{C}_1$ and $\mathrm{C}_2$ coming from properties (P1)-(P2). In other words, one can treat the original equation (under above considerations) as having a parabolic $p-$Laplacian behaviour.

Therefore, from Theorem \ref{ContCondp_laplace}, $w \in C^{\hat{\alpha}, \frac{\hat{\alpha}}{\hat{\theta}}}(Q^{-}_{\rho_0})$, with
$$
\hat{\alpha} = \frac{(pq-n)r-pq}{q[(p-1)r-(p-2)]} \geq \frac{(pq-n)r-pq}{q[(r-1)(m+p-2)+1]}.
$$
Additionally, notice that
$$
\begin{array}{rcl}
  \hat{\alpha} & \geq & \min\left\{\max\left\{\frac{\alpha^{-}_{\mathrm{Hom}}p}{p_m+\alpha_{\mathrm{Hom}}(m+p-3)}, \frac{2\alpha^{-}_{\mathrm{Hom}}(p-1)}{p_m(m+p-2)}\right\}, \frac{(pq-n)r-pq}{q[(p-1)r-(p-2)]}\right\} \\
   & \ge & \min\left\{\max\left\{\frac{\alpha^{-}_{\mathrm{Hom}}p}{p_m+\alpha_{\mathrm{Hom}}(m+p-3)}, \frac{2\alpha^{-}_{\mathrm{Hom}}(p-1)}{p_m(m+p-2)}\right\}, \frac{(pq-n)r-pq}{q[(r-1)(m+p-2)+1]}\right\} \\
   & = & \alpha.
\end{array}
$$

From Theorem \ref{ContCondp_laplace} we get for a constant $\mathrm{K}_0(\verb"universal")>0$
$$
\lvert\lvert w -w(0,0)\rvert \rvert_{L^{\infty}(Q^{-}_{\rho})}\leq \mathrm{K}_0\rho^{\hat{\alpha}} \quad \forall  \,\, 0<\rho<\frac{\rho_0}{2}.
$$

Moreover, since  $\alpha \le \hat{\alpha}$, we can conclude that
$$
\sup\limits_{(x,t)\in Q^{-}_{\mu\rho}} \lvert u(x,t)-u(0,0) \rvert \leq \mathrm{K}_0(\mu \rho )^{\alpha}, \quad \text{for any} \quad  0<\mu\rho<\mu\frac{\rho_{0}}{2}.
$$

Hence, relabelling
\begin{equation}\label{SecondEstThm}
\sup\limits_{(x,t)\in Q^{-}_{\rho}} \lvert u(x,t)-u(0,0) \rvert \leq \mathrm{K}_0\rho^{\alpha}, \quad  \text{for any} \quad 0<\rho<\mu\frac{\rho_{0}}{2}.
\end{equation}

  \item[({\bf 3})] If $\rho \in \left[\mu\frac{\rho_0}{2}, \mu\right)$:

  In this case, we obtain (by using Item (1) and \eqref{Estrho_0})
\begin{equation}\label{LastEstThm}
  \begin{array}{rcl}
\displaystyle  \sup\limits_{(x,t)\in Q^{-}_{\rho}} \lvert u(x,t)-u(0,0) \rvert & \le &  \displaystyle \sup\limits_{(x,t)\in Q_{\mu}} \lvert u(x,t)-u(0,0) \rvert \\
   & \le &  \left(\mathrm{C}+\frac{1}{4}\right)\mu^{\alpha} \\
   & \le &  \left(\mathrm{C}+\frac{1}{4}\right)\Big(\frac{2\rho}{\rho_{0}}\Big)^{\alpha} \\
   & \le & \left(\mathrm{C}+\frac{1}{4}\right)2^{\alpha}(16\gamma)^{\frac{p_m\alpha}{\alpha_0 \theta}}\rho^{\alpha}.
\end{array}
\end{equation}
\end{itemize}

Therefore, by selecting
{\scriptsize{
$$
  \mathrm{M}_0 \defeq \max\left\{\mathrm{C} + \frac{1}{4}, \,\mathrm{K}_0, \, \left(\mathrm{C}+\frac{1}{4}\right)2^{\alpha}(16\gamma)^{\frac{p_m\alpha}{\alpha_0 \theta}}\right\} = \max\left\{\mathrm{K}_0, \, \left(\mathrm{C}+\frac{1}{4}\right)2^{\alpha}(16\gamma)^{\frac{p_m\alpha}{\alpha_0 \theta}}\right\}
$$}}
and combining \eqref{FirstEstThm}, \eqref{SecondEstThm} and \eqref{LastEstThm}, we obtain \eqref{MainEstThm}, for any $\rho \in (0, \lambda)$.

Finally, a standard covering argument yields the desired estimate in any compactly supported sub-domain,
namely
$$
[u]_{C^{\alpha, \frac{\alpha}{\theta}}\left(Q^{-}_{\frac{1}{2}}\right)} \leq \mathrm{M}_0(\verb"universal"),
$$
thereby completing the proof.
\end{proof}

\section{An application to intrinsic Liouville results}

As an application of our main Theorem \ref{t1.2}, we prove a Liouville type result, which states that an entire weak solution of double degenerate parabolic equation
$$
\frac{\partial u}{\partial t}-\Div{(m\lvert u\rvert^{m-1}\lvert \nabla u\rvert^{p-2}\nabla u)}=0
$$
must be constant provided $|u(x, t)| = \text{O}(\lvert \lvert (x,t) \rvert\rvert^{\chi})$ for an appropriated semi-norm and exponent $\chi>0$ to be defined \textit{a posteriori}.

Before presenting such a classification result, let us define some notations: given $\beta \in (0,\alpha)$, where $\alpha \in (0,1)$ is the sharp regularity exponent in \eqref{Sharp_alpha}. Now, as before, let us consider the temporal scaling:
\begin{equation}\label{Scalbeta}
    \theta(\beta)=p(1-\beta)+\beta(3-m).
\end{equation}
For such a $\theta(\beta)$, consider the intrinsic parabolic cylinders:
$$
    \mathcal{C}^{\theta(\beta)}_{\tau}\colon= \Big(-\frac{1}{2}\tau^{\theta(\beta)},\frac{1}{2}\tau^{\theta(\beta)}\Big)\times B_{\tau}(0)
$$
and the intrinsic norm:
$$
    \lvert\lvert (x,t)\rvert\rvert_{\theta(\beta)} \colon= \lvert t \rvert^{{1/\theta(\beta)}}+\lvert x \rvert.
$$

Next result is the doubly degenerate version for the degenerate one in \cite[Theorem 1]{TU142}, which we will follow the lines of the proof.

\begin{theorem}
Let $u$ be an entire solution to
$$
   \mathcal{Q}_{m, p} u \defeq \frac{\partial u}{\partial t}-\Div{(m\lvert u\rvert^{m-1}\lvert \nabla u\rvert^{p-2}\nabla u)}=0 \quad \text{in} \quad \R^n \times \R.
$$
Suppose, for some $0<\beta<1$, there holds
\begin{equation}\label{Eq3.5}
    u(x,t)=\text{O}(\lvert\lvert (x,t)\rvert\rvert^{\beta}_{\theta(\beta)}) \quad \text{as} \quad \lvert\lvert (x,t)\rvert\rvert_{\theta(\beta)}\to \infty.
\end{equation}
Then, $u$ is constant.
\end{theorem}

\begin{proof} Fix a $\beta<\alpha<1$ (for $\alpha$ as in \eqref{Sharp_alpha}) and a large number $\mathrm{R}\gg 1$, and define $\mathcal{C}_{1}^{\theta(\beta)}\equiv \mathcal{C}_{1}$ the scaled function
$$
\mathcal{S}_{\mathrm{R}}(x,t)\colon= \frac{u(\mathrm{R} x,\mathrm{R}^{\theta(\beta)}t)}{\mathrm{R}^{\beta}}
$$

Firstly, we are going to show that
\begin{equation}\label{Eq3.6}
    \|\mathcal{S}_{\mathrm{R}}\|_{L^{\infty}(\mathcal{C}_{1})} \leq \mathrm{C} \quad  \text{for} \quad \mathrm{R}\gg 1.
\end{equation}

In effect, let $(x_{\mathrm{R}}, t_{\mathrm{R}}) \in \overline{\mathcal{C}}_{1}$ be a point which achieves the maximum, i.e.
$$
\|\mathcal{S}_{\mathrm{R}}\|_{L^{\infty}(\mathcal{C}_{1})} = \lvert \mathcal{S}_{\mathrm{R}}(x_{\mathrm{R}}, t_{\mathrm{R}})\rvert = \Bigg|\frac{u(\mathrm{R} x_{\mathrm{R}}, \mathrm{R}^{\theta(\beta)}t)}{\mathrm{R}^{\beta}} \Bigg|
$$

Now, notice that
$$
   \lvert\lvert (\mathrm{R} x_{\mathrm{R}}, \mathrm{R}^{\theta(\beta)}t_{\mathrm{R}})\rvert\rvert^{\beta}_{\theta(\beta)} \leq (2\mathrm{R})^{\beta}.
$$
Thus,

\begin{equation}\label{Eq3.7}
\frac{1}{2^{\beta}}\|\mathcal{S}_{\mathrm{R}}\|_{L^{\infty}(\mathcal{C}_{1})}=\Bigg|\frac{u(\mathrm{R} x_{\mathrm{R}}, \mathrm{R}^{\theta(\beta)}t)}{(2\mathrm{R})^{\beta}} \Bigg|  \leq \frac{|u(\mathrm{R} x_{\mathrm{R}}, \mathrm{R}^{\theta(\beta)}t)|}{\lvert\lvert (\mathrm{R} x_{\mathrm{R}}, \mathrm{R}^{\theta(\beta)}t_{\mathrm{R}})\rvert\rvert^{\beta}_{\theta(\beta)}}.
\end{equation}

In this point, by taking the limit as $\mathrm{R} \to \infty$, we must consider two cases:
\smallskip
\begin{enumerate}
  \item If $\lvert \lvert (\mathrm{R} x_{\mathrm{R}}, \mathrm{R}^{\theta(\beta)}t_{\mathrm{R}})\rvert\rvert_{\theta(\beta)}\to \infty$
\smallskip

In this case, the RHS in \eqref{Eq3.7} is bounded, by using the hypothesis \eqref{Eq3.5}. Thus,  the statement holds true.
\smallskip

  \item  If $\lvert \lvert (\mathrm{R} x_{\mathrm{R}}, \mathrm{R}^{\theta(\beta)}t_{\mathrm{R}})\rvert\rvert_{\theta(\beta)}$ remains bounded
\smallskip

Then, since $u$ is continuous, on compact sets we have
$$
\Bigg|\frac{u(\mathrm{R} x_{\mathrm{R}}, \mathrm{R}^{\theta(\beta)}t)}{(2\mathrm{R})^{\beta}} \Bigg| \leq \frac{\mathrm{C}}{(2\mathrm{R})^{\beta}} \to 0
$$
and one more time \eqref{Eq3.6} holds.
\end{enumerate}

Now, we observe that:
$$
\left\{
\begin{array}{rcl}
  \frac{\partial \mathcal{S}_{\mathrm{R}}}{\partial t}(x, t) & = & \mathrm{R}^{\theta(\beta) - \beta} \frac{\partial u}{\partial t}(\mathrm{R} x, \mathrm{R}^{\theta(\beta)} t) \\
  \Delta_{m, p} \mathcal{S}_{\mathrm{R}}(x, t) & = & \mathrm{R}^{[p(1-\beta)+\beta(1-m)+\beta]} (\Delta_{m, p} u )(\mathrm{R} x, \mathrm{R}^{\theta(\beta)} t).
\end{array}
\right.
$$

Hence,
\begin{equation}
    \frac{\partial \mathcal{S}_{\mathrm{R}}}{\partial t}-\Div{(\lvert \mathcal{S}_{\mathrm{R}}\rvert^{m-1}\lvert \nabla \mathcal{S}_{\mathrm{R}}\rvert^{p-2}\nabla \mathcal{S}_{\mathrm{R}})}=0 \quad  \text{in} \quad \mathcal{C}_{1}
\end{equation}

Now, for such a $\beta<\alpha<1$ and its corresponding $\theta(\alpha)$, we notice that, from the definition in \eqref{Scalbeta}, $\beta \mapsto \theta(\beta)$ is a decreasing function, hence $\theta_{\alpha}<\theta_{\beta}$. Therefore
, from Theorem \ref{t1.2}, we obtain
\begin{equation}\label{Eq3.10}
    \lvert \mathcal{S}_{\mathrm{R}}(x,t)-\mathcal{S}_{\mathrm{R}}(0,0)\rvert \leq \mathrm{C}\cdot \Big(\lvert t \rvert^{1/\theta(\alpha)}+\lvert x \rvert \Big)^{\alpha} \quad \forall (x,t) \in \mathcal{C}^{\theta(\alpha)}_{\frac{1}{2}}
\end{equation}

Hence, after scaling
$$
\begin{array}{rcl}
  \sup\limits_{\mathcal{C}^{\theta(\alpha)}_{\frac{\mathrm{R}}{2}}} \frac{\lvert u(x,t)-u(0,0)\rvert}{(\lvert t \rvert^{1/\theta(\alpha)}+\lvert x \rvert)^{\alpha}} & = & \sup\limits_{\mathcal{C}^{\theta(\alpha)}_{\frac{1}{2}}} \frac{\lvert u(\mathrm{R} x, \mathrm{R}^{\theta(\alpha)}t) -u(0,0)\rvert}{\mathrm{R}^{\alpha}(\lvert t \rvert^{1/\theta(\alpha)}+\lvert x \rvert)^{\alpha}} \\
   & = & \mathrm{R}^{\beta-\alpha}\sup\limits_{G^{\theta(\alpha)}_{\frac{1}{2}}} \frac{\lvert \mathcal{S}_{\mathrm{R}}(x,\mathrm{R}^{-\theta(\beta)+\theta(\alpha)}t)-\mathcal{S}_{\mathrm{R}}(0,0)\rvert}{(\lvert t \rvert^{1/\theta(\alpha)}+\lvert x \rvert)^{\alpha}} \\
   & \le & \mathrm{R}^{\beta-\alpha}\sup\limits_{G^{\theta(\alpha)}_{\frac{1}{2}}} \frac{\lvert \mathcal{S}_{\mathrm{R}}(x, \mathrm{R}^{-\theta(\beta)+\theta(\alpha)}t)-\mathcal{S}_{\mathrm{R}}(0,0)\rvert}{(\lvert \mathrm{R}^{-\theta(\beta)+\theta(\alpha)}t \rvert^{1/\theta(\alpha)}+\lvert x \rvert)^{\alpha}} \\
   & = & \text{o}(1) \quad \text{as} \quad \mathrm{R} \to \infty,
\end{array}
$$
by using \eqref{Eq3.10}, since by taking $\mathrm{R} \gg 1$ large enough, we can ensure $\mathrm{R}^{\theta(\alpha)-\theta(\beta)}<1$. Finally, we conclude $u\equiv u(0,0)$ in the whole $\mathbb{R}^n\times \mathbb{R}$.
\end{proof}

\section{Asymptotic estimates for problems ``close'' to heat equation}\label{SecAsymEst}

In this section we will prove that weak solutions of
\begin{equation}\label{Eq5.1}
  \frac{\partial u}{\partial t}-\Div(m\lvert u\rvert^{m-1}\lvert \nabla u \rvert^{p-2}\nabla u) = f(x, t) \quad \text{in} \quad \quad Q_1^{-},
\end{equation}
 become ``asymptotically smoother'' as long as our model case may be closer (in an appropriated manner to be clarified soon) to the homogeneous heat equation $\mathcal{H} u = \frac{\partial u}{\partial t} - \Delta u = 0$. For this reason, we will make use of a trick found in \cite{Stur17-2}, which consist in considering the following parameter:
\begin{equation}
\beta\colon= \frac{m-1}{p-1} \quad \text{for} \quad m\geq 1 \quad \text{and} \quad p\ge 2.
\end{equation}

In this point, we can re-write the equation \eqref{Eq5.1} as follows
\begin{equation}\label{Eq5.2}
\frac{\partial u}{\partial t}-\Psi(m,p,\beta)\Div{(\lvert \nabla u^{\beta +1}\rvert^{p-2}\nabla u^{\beta+1})}=f(x,t) \quad \text{in} \quad Q_1^{-},
\end{equation}
where $\Psi(m,p,\beta)=m\Big(\frac{1}{\beta +1}\Big)^{p-1}$.

From now on, we are going to consider the above formulation \eqref{Eq5.2} for our toy model. Moreover, it is worthwhile to note that such a constant $\Psi$ does not degenerate when we approach the parameters $m$ and $p$ to ones of heat equation, i.e.,
$$
  \Psi(m, p, \beta)\to 1 \quad  \text{as} \quad p \to 2^{+} \quad \text{and} \quad m\to 1^{+}.
$$

In the sequel, we derive a key approximation mechanism which provides a tangential path toward the regularity theory available for homogeneous heat operator.

\begin{lemma}[{\bf Caloric Approximation Lemma}]\label{CaloricLemma} Given $\delta>0$, there exist $\epsilon>0$ and $\epsilon^{\ast}>0$, depending only on $n$ and $\delta$, such that if $u$ is weak solution of \eqref{Eq5.2} with $\lvert\lvert u \rvert\rvert_{L^{\infty}(Q^{-}_{1})}\leq 1$, and
$$
\max\left\{\lvert p-2  \rvert, \lvert m-1 \rvert \right\}<\epsilon  \quad  \text{and} \quad  \lvert \lvert f \rvert\rvert_{L^{q,r}(Q^{-}_1)}\le \epsilon^{\ast},
$$
then we can find $w$ fulfilling
$$
\left\{
\begin{array}{rcrcl}
  \frac{\partial w}{\partial t}- \Delta w & = & 0 & \text{in} & Q^{-}_{\frac{1}{2}}\\
  w & = & u & \text{on} & \partial_p Q^{-}_{\frac{1}{2}}
\end{array}
\right.
$$
such that
$$
  \sup\limits_{Q^{-}_{\frac{1}{2}}}\lvert w - u\rvert \leq \delta.
$$
\end{lemma}

\begin{proof} Suppose for the sake of contradiction that the thesis of the Lemma is not true. This means that there would exist a $\delta_0>0$ and sequences
\[\begin{cases}
(p_{j})_j,(m_{j})_j, (u_{j})_j, (w_{j})_j\quad \text{and} \quad (f_{j})_j \\
\\
(\beta_{j})\colon= \frac{m_{j}-1}{p_{j}-1} \\
\\
\Psi(m_{j},p_{j},\beta_{j})\defeq m_{j}\Big(\frac{1}{\beta_{j}+1}\Big)^{p_{j}-1} \\
\\
v_{j}\colon= u_{j}^{\beta_{j}+1},
\end{cases}
\]
such that
\begin{center}
\[(*)\begin{cases}
\frac{\partial u_{j}}{\partial t}-\Psi(m_{j},p_{j},\beta_{j})\Div{(\lvert \nabla v_{j} \rvert^{p_{j}-2}\nabla v_{j})}=f_j  \quad \text{in} \quad Q^{-}_{1} \\
\lvert\lvert u_{j} \rvert\rvert_{\infty,Q_{1}}\leq 1 \\
\lvert \lvert f_{j} \rvert \rvert_{L^{q,r}(Q^{-}_{1})} <\frac{1}{j} \\
\lvert p_{j}-2 \rvert,\lvert m_{j}-1 \rvert <\frac{1}{j}
\end{cases}\]
\end{center}
and
\begin{equation}\label{HeatEquation}
  \left\{
\begin{array}{rclcl}
  \frac{\partial w_j}{\partial t}- \Delta w_j & = & 0 & \text{in} & Q^{-}_{\frac{1}{2}}\\
  w_j & = & u_j & \text{on} & \partial_p Q^{-}_{\frac{1}{2}}.
\end{array}
\right.
\end{equation}
However,
\begin{equation}\label{Eqw_j-u_j}
\sup\limits_{Q^{-}_{\frac{1}{2}}} \lvert u_{j}-w_j\rvert >\delta_{0} \quad \text{for every} \,\, j \in \mathbb{N}.
\end{equation}

Firstly, it is clear that
\begin{equation}\label{Lim_of_Const}
p_{j} \to 2, m_{j} \to 1 \quad \text{and} \quad \beta_{j} \to 1 \quad  \text{as}  \quad j \to \infty.
\end{equation}
 Now, as in the proof of Approximation Lemma \ref{l2.1}, up to a subsequence,
 \begin{equation}\label{Lim_of_u_j}
 u_{j} \to u_{\infty} \quad \text{in}  \quad Q^{-}_{\frac{1}{2}}.
 \end{equation}
 Moreover, in view of \eqref{Lim_of_Const} and \eqref{Lim_of_u_j}, and one more time arguing as in the proof of Approximation Lemma \ref{l2.1} we can appeal to a stability result, as the one in \cite{KinPar10}, in order to pass the limit in the equation satisfied by $u_{j}$ to conclude that $u_{\infty}$ is a weak solution of
$$
\frac{\partial \mathfrak{h}}{\partial t}- \Delta \mathfrak{h}  =  0 \quad \text{in} \quad Q^{-}_{\frac{1}{2}}.
$$

Now, we focus our attention in \eqref{HeatEquation}. Form Maximum Principle we have
$$
\|w_j\|_{L^{\infty}\left(Q^{-}_{\frac{1}{2}}\right)}\leq \|u_j\|_{L^{\infty}\left(\partial_p Q^{-}_{\frac{1}{2}}\right)} \leq \|u_j\|_{L^{\infty}\left(Q^{-}_{1}\right)} \leq 1.
$$
From regularity estimates to heat operator (see \cite{K08}) we can as above to assure that $w_j \to w_{\infty}$ in $Q^{-}_{\frac{1}{2}}$. Furthermore, $w_{\infty}$ is a weak solution to
$$
  \left\{
\begin{array}{rclcl}
 \frac{\partial \mathfrak{h}}{\partial t}- \Delta \mathfrak{h} & = & 0 & \text{in} & Q^{-}_{\frac{1}{2}}\\
  \mathfrak{h} & = & u_{\infty} & \text{on} & \partial_p Q^{-}_{\frac{1}{2}}.
\end{array}
\right.
$$
By the uniqueness of the Dirichlet problem, we conclude that $u_{\infty} = w_{\infty}$.

Finally, for a large enough $j$, we obtain
$$
\lvert u_{j} - w_{j} \rvert \leq \lvert u_{j}-u_{\infty}\rvert + \lvert w_{j} - u_{\infty}\rvert \leq \frac{\delta_{0}}{2}+\frac{\delta_{0}}{2}=\delta_{0} \quad \text{in} \quad Q^{-}_{\frac{1}{2}},
$$
which clearly yields a contradiction with \eqref{Eqw_j-u_j}. This completes the proof.
\end{proof}

Finally, we are in a position to establish asymptotic regularity estimates provided the parameters in the equation approach the linear case.

\begin{theorem}\label{Mthm2}\label{ThmAsymReg} Let $u$ be a bounded weak solution of \eqref{Eq5.1}, where $f \in L^{q, r}(Q^{-}_1)$  and
$$
\frac{1}{r}+\frac{n}{pq}<1 \quad \text{and} \quad  \frac{3}{r}+m\left(1-\frac{1}{r}\right)+\frac{n}{q}\le 2
$$
are in force. Given, $\alpha \in(0,1)$, there exists an $\varepsilon_0>0$ such that if $\max\left\{m-1, p-2\right\} < \varepsilon_0$, then any solution of \eqref{Eq5.1} belongs to $C^{\alpha, \frac{\alpha}{\theta}}$ (in parabolic sense). Moreover, there exists a constant $\mathrm{M}_0(\verb"universal")>0$ such that
$$
  \displaystyle  [u]_{C^{\alpha, \frac{\alpha}{\theta}}\left(Q^{-}_{\frac{1}{2}}\right)} \leq \mathrm{M}_0.\left[\|u\|_{L^{\infty}(Q^{-}_1)} + \|f\|_{L^{q, r}(Q^{-}_1)}\right].
$$
Quantitatively, such an estimate states that $u \in C^{1^{-}, {\frac{1}{2}}^{-}}\left(Q^{-}_{\frac{1}{2}}\right)$.
\end{theorem}

\begin{proof} With the Caloric approximation device, i.e. Lemma \ref{CaloricLemma}, in hands, and making use of available $C_{\text{loc}}^{1, {\frac{1}{2}}}$ regularity estimates for caloric profiles $\frac{\partial \mathfrak{h}}{\partial t}- \Delta \mathfrak{h}  =  0$ (see \cite[Ch.2, \S 3]{Evans10}), we can proceed similarly as in the proof of Theorem \ref{t1.2}. For this reason, we will omit the details here.
\end{proof}

\section{Final comments  and further connections to related problems}

Finally, let us present a family of operators where our results take place.

\begin{example} A typical class of operators fulfilling the assumptions (P1)-(P3), and which we
can apply our result it is given by
$$
(\Omega_T, \mathbb{R}, \mathbb{R}^n)\ni (x, t, s, \xi)\mapsto \mathcal{A}(x, t, s, \xi) = \mathfrak{a}(x, t)|s|^{m-1}|\xi|^{p-2}\xi, \quad \text{for} \quad m \ge 1, p\ge 2,
$$
where $\mathfrak{a}$ is a function  bounded away from zero and infinity and satisfying
$$
|\mathfrak{a}(z)-\mathfrak{a}(z_0)| \leq \omega(|z-z_0|)
$$
for a universal modulus of continuity $\omega: [0, \infty) \to [0, \infty)$. At least, the case $\mathfrak{a} = \text{const}.$ one reduces to the well-known $(m, p)-$Laplacian operator.
\end{example}

In conclusion, as a consequence of our findings, we are able to address regularity issues for non-negative bounded weak solutions of doubly nonlinear equations (a sort of sub-linear Trudinger's equation $\mathrm{k}\leq \min\{p-1, 1\}$) as follows
\begin{equation}\label{EqDNEq}
  \frac{\partial (u^{\mathrm{k}})}{\partial t}-\Div(|\nabla u|^{p-2}\nabla u) = f(x, t) \quad \text{in} \quad \Omega_T, \,\,\,\text{for} \quad p > 2 \quad \text{and} \quad  \mathrm{k} \in (0, 1),
\end{equation}

Such an equation exhibits an interesting feature: it is singular in time, since $u^{\mathrm{k}-1}$ blows up at those points where $\{u=0\}$, and it is degenerate in space, since the modulus of ellipticity, i.e. $|\nabla u|^{p-2}$, collapses at those points where $\{|\nabla u|=0\}$.

In this scenario, we provide one more contribution to the study of such a class of PDEs, which model turbulent filtration of non-Newtonian fluids through a porous media, see \cite{DT94} for an enlightening manuscript on this topic, and \cite{HL13}, and \cite{DiUrb20} and \cite{KSU12} for related regularity results for the homogeneous problem and Trudinger's equation.

In effect, equations like \eqref{EqDNEq} are equivalent to the ones like \eqref{1.2} provided weak solutions are strictly away from zero. As a matter of fact, by performing the change $v = u^{\mathrm{k}}$, equation \eqref{EqDNEq} reads
$$
   \frac{\partial v}{\partial t}-\Div\left(\frac{1}{\mathrm{k}^{p-1}}v^{\frac{(1-\mathrm{k})(p-1)}{\mathrm{k}}}\lvert \nabla v \rvert^{p-2}\nabla v\right) = f(x, t) \quad \text{in} \quad \Omega_T.
$$
In this setting, we have $m = \frac{(1-\mathrm{k})(p-1)}{\mathrm{k}}+1\ge 1$ for $p > 2$ and $0< \mathrm{k}< 1$.

Therefore, we can access the sharp/improved regularity estimates available in Theorem \eqref{t1.2}, thereby establishing corresponding ones for model cases as \eqref{EqDNEq}. Precisely, weak solutions of \eqref{EqDNEq} $u \in C_{\text{loc}}^{\alpha, \frac{\alpha}{\theta}}(\Omega_T)$, where
$$
  \alpha \defeq \min\left\{\max\left\{\frac{\mathrm{k}p\alpha^{-}_{\mathrm{Hom}}}{\mathrm{k}p_m + \alpha_{\mathrm{Hom}}(p-1-\mathrm{k})}, \frac{2\mathrm{k}\alpha^{-}_{\mathrm{Hom}}}{p_m} \right\},\, \,\frac{\mathrm{k}[(pq-n)r-pq]}{q[(r-1)(p-1)+\mathrm{k}]}\right\}
$$
and
$$
\theta \defeq p -\alpha\left(\frac{p-1}{\mathrm{k}}\right).\Big(1-\frac{\mathrm{k}}{p-1}\Big),
$$
where $\alpha_{\mathrm{Hom}}>0$ was recently addressed in \cite[Theorem 2.3]{HL13}.

Finally, as in Theorem \ref{ThmAsymReg} the previous estimates are stable when the parameter $\mathrm{k}$ goes to $1^{-}$, thereby recovering the sharp estimates for the evolutionary $p-$Laplacian problem addressed in \cite[Theorem 3.4]{TU14}.

\bigskip

\noindent{\bf Acknowledgements.} This manuscript is part of the second author's Ph.D thesis. He would like to thank the Department of Mathematics at Universidade Federal do Cear\'{a} for fostering a pleasant and productive scientific atmosphere, which has benefited a lot the final outcome of this project. E.C. Bezerra J\'{u}nior thanks to Capes-Brazil (Doctoral Scholarship). J.V. da Silva and G.C. Ricarte have been partially supported by Conselho Nacional de Desenvolvimento Cient\'{i}fico e Tecnol\'{o}gico (CNPq-Brazil) under Grants No. 310303/2019-2 and No. 303078/2018-9.

\end{document}